\documentclass[sn-mathphys,Numbered]{sn-jnl}

\usepackage{graphicx}%
\usepackage{multirow}%
\usepackage{amsmath,amssymb,amsfonts}%
\usepackage{amsthm}%
\usepackage{mathrsfs}%
\usepackage[title]{appendix}%
\usepackage{xcolor}%
\usepackage{textcomp}%
\usepackage{manyfoot}%
\usepackage{booktabs}%
\usepackage{algorithm}%
\usepackage{algorithmicx}%
\usepackage{algpseudocode}%
\usepackage{listings}%

\usepackage{dsfont}
\usepackage{mathtools}
\usepackage{bm}

\theoremstyle{thmstyleone}%

\theoremstyle{thmstyletwo}%
\newtheorem{remark}{Remark}%

\theoremstyle{thmstylethree}%
\newtheorem{definition}{Definition}%

\let\mib=\boldsymbol
\usepackage{amsmath, amsthm, amssymb, enumerate}
\usepackage{color}
\usepackage{soul}
\usepackage{datetime}
\usepackage{fancyhdr}
\usepackage{caption}
\usepackage{graphicx}

\usepackage{geometry}
\usepackage{eurosym}%
\usepackage{ulem}
\usepackage{pdfpages}
\usepackage{amssymb,amsthm,amsmath,mathtools}

\newtheorem{theorem}{Theorem}[section]

\newtheorem{corollary}[theorem]{Corollary}
\newtheorem{rem}{Remark}

\usepackage{float}
\usepackage{subfigure}
  \def\startnewsection#1#2{\section{#1}\label{#2}\setcounter{equation}{0}}   
\def\sgn{\mathop{\rm sgn\,}\nolimits}    
\def\mx{{\bf x}}
\def\my{{\bf y}}
\def\R{\mathbb{R}}
\def\eps{\varepsilon}
\def\qty{\displaystyle}
\def\meta{{\mib \eta}}

\def\tilde{\widetilde}
\def\theequation{\thesection.\arabic{equation}}

\newcommand{\pa}{\partial}

\begin{document}

\title[Shock Filter Cahn-Hilliard Equation: Regularized to Entropy Solutions]{Navigating the Complex Landscape of Shock Filter Cahn-Hilliard Equation: From Regularized to Entropy Solutions}


\author[1,2]{\fnm{Darko} \sur{ Mitrovic}}\email{darko.mitrovic@univie.ac.at}

\author*[1,3,4]{\fnm{Andrej} \sur{Novak}}\email{andrej.novak@phy.hr}
\equalcont{These authors contributed equally to this work.}
%
%
\affil*[1]{\orgdiv{Department of Mathematics}, \orgname{University of Vienna}, \orgaddress{\street{Oskar Morgenstern Platz-1}, \city{Vienna}, \postcode{1090}, \state{Austria}, \country{Austria}}}
\affil[2]{\orgdiv{Department of Mathematics and Natural Sciences}, \orgname{University of Montenegro}, \orgaddress{\street{Cetinjski put bb}, \city{Podgorica}, \postcode{81000}, \state{Montenegro}, \country{Montenegro}}}
\affil*[3]{\orgdiv{Department of Physics, Faculty of Science}, \orgname{University of Zagreb}, \orgaddress{\street{Bijenicka cesta 32}, \city{Zagreb}, \postcode{10000}, \state{Croatia}, \country{Croatia}}}

\affil*[4]{\orgdiv{Dubrava University Hospital}, \orgaddress{\street{Avenija Gojka Šuška 6}, \city{Zagreb}, \postcode{10000}, \state{Croatia}, \country{Croatia}}}

%

\abstract{
Image inpainting involves filling in damaged or missing regions of an image by utilizing information from the surrounding areas. In this paper, we investigate a fully nonlinear partial differential equation inspired by the modified Cahn-Hilliard equation. Instead of using standard potentials that depend solely on pixel intensities, we consider morphological image enhancement filters that are based on a variant of the shock filter:
\begin{align*}
\partial_t u &= \Delta \left(-\nu \arctan(\Delta u)|\nabla u| - \mu \Delta u \right)+ \lambda(u_0 - u).
\end{align*}
This is referred to as the Shock Filter Cahn-Hilliard Equation. The equation is nonlinear with respect to the highest-order derivative, which poses significant mathematical challenges. 
To address these, we make use of a specific approximation argument, establishing the existence of a family of approximate solutions through the Leray-Schauder fixed point theorem and the Aubin-Lions lemma.
In the limit, we obtain a solution strategy wherein we can prove the existence and uniqueness of solutions. Proving the latter involves the Kruzhkov entropy type-admissibility conditions.

Additionally, we use a numerical method based on the convexity splitting idea to approximate solutions of the nonlinear partial differential equation and achieve fast inpainting results. To demonstrate the effectiveness of our approach, we apply our method to standard binary images and compare it with variations of the Cahn-Hilliard equation commonly used in the field.}




\maketitle

%
%
%










\startnewsection{Introduction}{sec:introduction}
Digital image inpainting is an important task in image processing applications that involves reconstructing unknown regions of an image based on information from its surrounding areas. We will refer to these unknown regions as the inpainting domain. For a given image $u_0$, defined on the image domain $\Omega \subset \mathbb{R}^2$, the objective is to restore the image on the inpainting domain  $\omega \subset \Omega$ so that the reconstructions are not easily detectable by an ordinary observer.

Drawing inspiration from manual inpainting techniques, Bertalmio et al. \cite{Belt00} were among the first to introduce an anisotropic PDE that smoothly propagates information from outside the inpainting domain  $\Omega\setminus \omega$ inside the domain $\omega$, following the direction of isophotes (contours of constant greyscale intensity). Anisotropic diffusion is a natural choice as it ensures the sharpness of the reconstructed image. Furthermore, they incorporated the isophote direction as a boundary condition at the edges of the inpainting domain. More precisely, this means that the equation
\begin{align}
u_t = -\nabla^\perp u \cdot \nabla\Delta u,
\end{align}
must be solved on $\omega$, where the perpendicular gradient  $\nabla^\perp  u = (-u_y,u_x)$ is the direction of smallest change in image intensity and $\Delta u$ is a measure of image smoothness. Because isophotes depend on the geometry of the inpainting domain and can intersect over time, it became necessary to intertwine this process with a diffusion procedure.

As discussed in \cite{Bert01}, this approach borrows ideas from computational fluid dynamics and applies them to image analysis challenges. One can visualize image intensity as a stream function for a two-dimensional incompressible flow. In this analogy, the Laplacian of the image intensity acts as the fluid's vorticity and is transported into the inpainting domain by a vector field defined by the stream function. Such an interpretation underscores the need for incorporating diffusion into the inpainting problem. 
Finally, this approach is designed to preserve isophotes while ensuring that the gradient vectors match at the boundary of the inpainting domain. This method is based on the Navier-Stokes equations of fluid dynamics, taking advantage of well-established theoretical and numerical results in that field. Since the success of the inpainting process depends on human observers' perceptions of edges, color, and texture, it is reasonable to model the missing edges using elastica-type curves \cite{Shen03, Than21, Ring18}.

More precisely, if we extend the inpainting domain $\omega$ in $\Omega$ and denote it by $\omega_\epsilon$, then we can extrapolate the isophotes of an image $u$ by a collection of curves $\{\gamma_t\}_{t \in [I_0, I_m]}$ with no mutual crossing, which coincides with the isophotes of $u$ on $\omega_\epsilon \setminus \omega$ and minimize the energy 
\begin{align}\label{third}
\int_{I_0}^{I_m}\!\!\int_{\gamma_t}(\alpha + \beta|\kappa_{\gamma_t}|^p) dsdt, \ \ \alpha,\beta\in \R^+,
\end{align}
where $[I_0, I_m]$ is the intensity span and $\kappa_{\gamma_t}$ is the curvature of $\gamma_t$. 
By adjusting parameters $\alpha$ and $\beta$ based on the specific application at hand, this energy effectively penalizes a generalized form of Euler's elastica energy.

In \cite{Shen03}, the authors introduced two inpainting models. The first is based on the seminal Mumford-Shah image model \cite{Mumf89}, while the second—termed the Mumford-Shah-Euler image model—serves as its high-order correction. This enhanced model augments the original by substituting the embedded straight-line curve model with Euler's elastica, a concept initially presented by Mumford in the context of curve modeling. Although this approach provides a more refined representation, it comes with a computational cost. Recognizing this challenge, efforts have been directed towards devising more efficient algorithms \cite{BS, Tai11, Zhu13}.

This type of model falls under variational inpainting, which incorporates the total variation inpainting \cite{Rudi92, Esed04, Chan06} and the active contour model based on the well-known Mumford and Shah segmentation \cite{Tsai01}. In particular, Esedoglu and Shen \cite{Esed02} introduced a modification of the Mumford–Shah segmentation model \cite{Shen02, Tsai01} that incorporates the curvature of edge contours into the functional. The curvature term provides control over both the position and direction of edges in the digital image. For a given image $u_0(x)$, one needs to minimize
\begin{align} \label{Ese}
MSE(u,K) = \int_{\Omega\setminus K} |\nabla u|^2 dx + \int_K (\alpha + \beta \kappa^2)ds + \lambda_0\int_{\Omega\setminus\omega}(u - u_0)^2dx, \ \ \alpha,\beta,\lambda_0\in \R^+,
\end{align}
where the unknown set $K$ represents the edge collection, i.e., the union of curves that approximate the edges of the given image $u_0(x)$, \( \kappa \) is the curvature, and  $ds$ is the length element. The last term in \eqref{Ese} is called the fidelity term and it penalizes deviations, in the $L^2$ sense, of the piecewise smooth function $u$ from the original image $u_0$ outside the inpainting domain. Several years later, Bertozzi et al. \cite{Bert07, Bert07d} proposed a binary image inpainting model that simplifies the previously presented Esedoglu–Shen model \eqref{Ese}. They also observed that the fourth-order gradient flow in the Esedoglu–Shen model shared similarities with a much simpler model called the Cahn–Hilliard equation.

\subsection{The Cahn-Hilliard equation}
The Cahn-Hilliard equation (also CHE in the sequel) is a macroscopic field model that describes the phase separation \cite{Cate18, Novi08} of a binary alloy at a fixed temperature. Beyond its primary application, the CHE is used in diverse domains, including pattern formation \cite{Zaks05, Zhao20}, biology \cite{Fis}, fluid dynamics \cite{Liu03, Han15, Maga13}, and most notably in image processing \cite{Burg09, Bosc14, Bosc15, Brki20, Brki18}.
\noindent More precisely, let $u_0: \Omega \subset \mathbb{R}^2 \to \{-1,1\}$ be a given binary image, and suppose that $\omega \subset \Omega$ denotes the inpainting domain. 
The interpolation $u: (0, T) \times \Omega \to \{-1,1\}$ of the original image $u_0$  is obtained as the solution of the modified CHE:
\begin{align}
\label{che-old}
u_t &= \Delta\left(\nu H'(u) - \mu\Delta u\right) + \lambda(\mx)(u_0 - u), \;\; \text{on} \;\;  (0, T) \times \Omega,\\
\label{ic-bc-1}
u\big|_{t=0} &= u_0(\mx), \ \ 
\end{align} together with the appropriate boundary conditions that will be discussed later. Here,  the constants $\mu, \nu > 0$ depend on a concrete situation, and for some large $\lambda_0 > 0$, the fidelity term $\lambda: \Omega \to \mathbb{R}$  is defined as follows
\begin{align}
\label{lam}
\lambda(\mx) = \begin{cases} \lambda_0, \;\;\; \text{if } \mx\in \Omega \setminus\omega, \\  0, \;\;\; \text{if } \mx\in \omega. \end{cases}
\end{align}


In the special case when $H(u) = \frac{1}{2}(1 - u^2)^2$ (termed the double-well potential), equation \eqref{che-old} reduces to the well-known CHE, as detailed in the original paper \cite{Cahn}. Before delving into various forms of $H$, let us briefly comment on the fidelity term $\lambda$. 

The fidelity term plays a crucial role in the equation, ensuring that the inpainted image $u$ closely resembles the original image $u_0$ outside the damaged areas.
Indeed, for large $\lambda_0$, the term $\lambda(\mathbf{x})(u - u_0)$ must not be large, or it cannot be controlled by the rest of the equation. Therefore, for large $\lambda_0$, the inpainted image $u$ must be close to the original image in the image domain $\Omega \setminus \omega$ (since the product $\lambda_0(u - u_0)$ is moderate and, roughly speaking, can be controlled by $u_t - \Delta\left(-\nu H'(u) - \mu\Delta u\right)$; see also \eqref{Ese}). We note in passing that this is also the reason why the boundary data for \eqref{che-old} are not of substantial importance and it will have an influence only in a very small neighborhood of the boundary (assuming that the inpainted domain is away from the boundary). Moreover, the presence of $\lambda$ ensures that the form of the image has no influence on the inpainting away from the inpainting domain and we can always extend the image domain and force boundary conditions of our choice. 
Let us now discuss choices of $H$.

\subsection{The choice of nonlinear potential $H$}
The functional $H(u)$ represents a nonlinear potential with two wells corresponding to $u=-1$ and $u=1$. In the context of inpainting binary images, the values $u=1$ and $u=-1$ denote the white and black pixels, respectively. In short, since $H$ drives the pixel intensities toward one of these wells, one could, borrowing from machine learning terminology, view $H$ as a form of classifier.
In the recent literature, the search for the physically meaningful potential $H$ has been extensive. The initial motivation for this arose from a simple observation that the solutions $u$ of the CHE equipped with double-well potential do not stay bounded to the interval $[-1,1]$. In many applications, this violates the inherent physical laws governing the system. 
However, from a computational standpoint, one can address this issue using thresholding. These methods redefine any values of $u > 1$ as $u = 1$ and any values of $u < -1$ as $u = -1$. For example, this approach can be found in the research by Cherfils et al. \cite{Cher15d}, wherein the numerical scheme is paired with thresholding in its final steps.

Another possibility is to utilize the non-smooth logarithmic potentials \cite{Cher11, Cope92}:
\begin{align}\label{non-smooth}
H_{log}(u) = \frac{\theta}{2}\left ( (1+u)\log(1+u) + (1 - u) \log(1 - u)  \right ) +  \frac{\theta_c}{2}(1 - u^2),
\end{align}
where $0< \theta <\theta_c$ is a constant. The singularity of $H_{log}'(u)$ furthermore imposes the bounds on $u \in (-1,  1)$.
The application of a non-smooth logarithmic potential in image inpainting tasks has been demonstrated to decrease the convergence time \cite{Cher15}.

In recent years, a version of the CHE that uses a non-smooth potential has been studied in works such as \cite{Bosc14, Garc18} for image inpainting problems. This approach involves using a non-smooth double obstacle potential, defined as:
\begin{align}\label{non-smooth-1}
H_{quad}(u) = \psi(u) + \frac{1}{2}(1 - u)^2, \;\;\;\;\;\; \psi(u) = I_{[-1,1]}(u) =  \begin{cases} 0, \;\;\; \text{if } u \in [-1,1],\\   +\infty,  \;\;\; \text{otherwise,}\end{cases}
\end{align}
where $I_A$ is the indicator function of the set $A$. In \cite{Bosc14}, authors developed an efficient finite-element solver for image inpainting problems and demonstrated that non-smooth potential produces superior results when compared to the double well potential.

At the end of this subsection, we note that a well-posedness result for a more general situation of \eqref{che} with so-called degenerate diffusion mobility (i.e. the diffusion coefficient depends on $u$ as well) was obtained in \cite{DD}.   

\subsection{Numerical solutions based on the convexity splitting approach}
An additional advantage of this approach is the availability of efficient numerical methods for solving the CHE \cite{Bosc14, Chen08, Bert11}. Here, we will employ ideas based on the convexity splitting introduced by Eyre \cite{Eyre98}, and later investigated by several researchers \cite{Chen08, Glas16, Gome11, Shin17}. In the convexity splitting scheme, the energy functionals are divided into convex and concave parts. The main idea is to treat the convex part implicitly and the concave part explicitly.
Namely, let $\mathcal{H}$ be a Hilbert space and $\mathcal{E} = \mathcal{E}_1 - \mathcal{E}_2$ energy such that it could be written as the difference of two convex energies $\mathcal{E}_1$ and $\mathcal{E}_2$. Then, the discretization
\begin{align}\label{fede}
\frac{u^{n+1} - u^n}{\Delta t} = -\nabla_\mathcal{H} \mathcal{E}_1(u^{n+1}) - \nabla_\mathcal{H} \mathcal{E}_2(u^{n}), 
\end{align}
of the gradient flow 
\begin{align}\label{gradfl}
u_t = -\nabla_\mathcal{H} \mathcal{E}(u), 
\end{align}
satisfies the energy stability property  $\mathcal{E}(u^{n+1}) \leq \mathcal{E}(u^{n})$. Moreover, under the appropriate conditions \cite{Rain20}, this approach can allow for an unconditionally stable time-discretization scheme.
The CHE is derived as the $H^{-1}$ gradient flow of the Ginzburg--Landau energy defined in the following way:
\begin{align}\label{chee}
\mathcal{E}_1[u] = \int_{\Omega} \left[\nu H(u(\mx)) + \frac{1}{2} \mu |\nabla u(\mx)|^2\right]d\mx,
\end{align}
where $u$ typically denotes the concentration, $H(u)$ is the Helmholtz free energy density, 
$\varepsilon$ is proportional to the thickness of the transition region between two phases. For further insights into the derivation of both first and second-order unconditionally energy-stable numerical schemes based on equation \eqref{chee}, the work in \cite{Rain20} serves as a comprehensive reference.

In addition, one could derive the fidelity term as the $L^2$ gradient flow of energy defined by
\begin{align}\label{fede-1}
\mathcal{E}_2[u] = \frac{\lambda}{2}\int_{\Omega \setminus \omega} \left(u_0 - u\right)^2 d\mathbf{x}.
\end{align}
Let us note that the modified CHE \eqref{che-old} is not strictly a gradient flow but can be interpreted as the superposition of gradient descent with respect to the $H^{-1}$ inner product for the energy defined by \eqref{chee} and the gradient descent with respect to the $L^2$ inner product defined by \eqref{fede-1}.

As we will see in Section 4, the numerical method based on convexity splitting can be applied to the modified CHE \eqref{che-old} by applying it separately to the functionals $\mathcal{E}_1$ and $\mathcal{E}_2$ separately.

\subsection{The current contribution}
Recently, in \cite{Nova22}, a regularized version of the shock filters  $-\sgn ({\Delta u}) |\nabla u|$ developed by Rudin and Osher \cite{Osh} was considered as a potential choice for the generalization of the functional $H$. Here, $ -\sgn ({\Delta u})$  is an edge-detection term that exhibits a changing sign across any essential singular feature so that the local flow field is governed toward the features, and $|\nabla u|$ is the magnitude of the gradient.
The main idea revolves around utilizing a dilation process near image maxima and an erosion process around image minima. To determine the influence zone of each maximum and minimum, we employ the Laplacian as a classifier, borrowing terminology from machine learning. If the Laplacian is negative, the pixel in question is assigned to the influence zone of a maximum; conversely, if the Laplacian is positive, the pixel is considered to belong to the influence zone of a minimum. When applied to the digital image, this equation creates strong discontinuities at image edges and a piecewise constant segmentation within regions of similar grayscale intensities. 
Shock filter belongs to the class of morphological image enhancement equations \cite{Bua06, Gil04, Cal10, Sim21}. 
From the mathematical point of view, this equation satisfies a maximum–minimum principle which means that the pixel intensities of the processed image will remain within the range of the original image.  Besides the mathematical analysis of the regularized equation, the authors \cite{Nova22} demonstrated that this equation shows the tendency to extend image structures while preserving image sharpness, therefore eliminating the diffusive effects visible in the case of a CHE with a double-well potential. However, in certain examples, this modification of CHE has demonstrated an uncritical tendency to form edges, and additional steps are needed to control this effect. In this paper, we will address this issue by proposing a more regular classifier $-|\nabla u| \arctan({\Delta u})$ that, in applications, has a less expressed tendency towards edge formation. The problem to be considered here thus has the form
\begin{align}
\label{che}
\pa_t u &= \Delta \Big(-\nu \arctan(\Delta u )|\nabla u| - \mu \Delta u \Big)+ \lambda(u_0 - u) \;\; \text{on} \;\;  (0, T) \times \Omega,\\
\label{ic}
u(t,\mx) &= u_0(\mx) \;\; \text{on} \;\;  \{t = 0\} \times \Omega,\\
\label{bc}
u(t,\mx)&= \frac{\pa u(t,\mx)}{\pa \vec{n}} = 0   \;\; \text{on} \;\;  [0, T]\times \partial\Omega .
\end{align} 
We denote this as the Shock Filter Cahn-Hilliard equation. In addition to conducting numerical simulations on standard inpainting examples, we will undertake a comprehensive mathematical analysis of the equation. This analysis will include demonstrating the compactness of approximate solutions using the Leray-Schauder fixed-point theorem and the Aubin-Lions lemma. Through this process, we will establish a solution framework within which we can prove the stability of solutions to this equation. The framework bears resemblance to the entropy admissibility conditions introduced in Kruzhkov's work \cite[Definition 1]{Kru}, hence we term functions satisfying the equation within this framework "entropy solutions."  It is worth noting that, unlike the entropy admissibility concept from Kruzhkov's work, we utilize entropies of the form $\eta(u)=(u-\varphi)^2$, where $\varphi\in Lip([0,T];H^2(\Omega))$.

\subsection{Mathematical aspects}
The variants of the CHE presented in \eqref{che-old} hold not only significant practical importance but also a rich mathematical structure. Numerous existence and uniqueness results can be found for the equation. Typically, the latter demonstrates that the equation has an appropriate physical background; processes in nature are invariably predictable under fixed conditions. This also suggests that different numerical methods should converge towards the same result.

Regarding previous work on the subject, an exhaustive overview can be found in \cite{Mir_book}. Roughly speaking, it is always possible to establish local existence and uniqueness for \eqref{che-old}, but global existence depends on the growth rate of the potential $H$. Besides the standard double-well potential for $H$, we particularly mention the situation of the logarithmic type potential, which is regular inside the interval $(-1,1)$ and was introduced in \cite{Cope92}. In \cite{Cher15}, one can find proof of the local well-posedness, while in \cite{MirAIMS}, the author presents an intricate proof of the global well-posedness. The concept of the so-called double obstacle potential was first introduced in \cite{BE} and subsequently explored in the context of the inpainting problem in \cite{Garc18}. In these studies, the authors established the well-posedness of the problem, demonstrating that the function $u$ satisfying the problem meets a set of equations and an inequality.

In the current work, we encounter a situation somewhat akin to that in \cite{BE, Garc18}. However, instead of dealing with an irregular coefficient, we confront the presence of a second-order derivative term under a nonlinear function. Consequently, we find it necessary to formulate a generalization of the standard weak solution concept and involve corresponding admissibility conditions to ensure the well-posedness of our problem. These requisites are consolidated in Definition \ref{def-admissibility} and termed entropy solutions. It is worth noting that the inequality outlined in Definition \ref{def-admissibility} corresponds to \cite[(1.16)]{BE} and \cite[(1.17)]{BE}, or alternatively, \cite[(1.5a)]{Garc18} and \cite[(1.5b)]{Garc18}.

As for the existence proof, our first step is to approximate the operator on the right-hand side of \eqref{che} using convolution operators at suitable points. Then, via the Leray-Schauder fixed-point theorem, we aim to establish the existence of a family of solutions, denoted as $(u_\epsilon)$, corresponding to the regularized equation. By letting the regularization parameter go to zero and using the fact that $\arctan$ is an increasing function, we reach an inequality that will represent our solution concept.

We note that this is not an unusual situation, as one can infer from the entropy solution concept in \cite[Definition 1]{Kru} or the viscosity solution concept in \cite[Definition 2.2]{Lio}. However, in both of the mentioned approaches, a smooth entropy or viscosity solution is also a solution in the classical sense. This is not the case with our solution framework, but to some extent, it can be overcome with the help of Young measures \cite[Theorem 1]{DiPM}. We shall comment on this in more detail at the end of Section 3.

Regarding the boundary conditions, we have already explained that from a practical point of view, they are not of substantial importance. However, this could be an obstacle in the theoretical analysis of the equation. We note that \eqref{che}  augmented with the Dirichlet conditions
\begin{equation}
\label{dirichlet}
u(t,\cdot)\big|_{\pa \Omega}=\Delta u(t,\cdot)\big|_{\pa \Omega}=0,
\end{equation} or the Neumann conditions
\begin{equation}
\label{neumann}
\frac{\pa u(t,\cdot)}{\pa \vec{n}}\Big|_{\pa \Omega}=\frac{\pa (\Delta u(t,\cdot))}{\pa \vec{n}}\Big|_{\pa \Omega}=0,
\end{equation} 
are handled similarly since we can construct approximate solutions via the Galerkin approximation (see Section 3.2). A problem here lies in the fact that neither of the latter conditions (Dirichlet or Neumann) allows for the elimination of boundary terms in \eqref{che} when deriving the energy estimate. Namely, informally speaking, if we multiply \eqref{che} by $u$ and integrate over $[0,T] \times \Omega$, upon application of integration by parts, some boundary terms will remain on the right-hand side. Moreover, it is not clear what the boundary conditions mean since we do not have the existence of strong boundary traces for granted. Indeed, due to the presence of the nonlinear term $\arctan(\Delta u)$, the solution $u$ might not be of bounded variation. Therefore, we shall (implicitly) include the boundary conditions in the variational formulation of \eqref{che}, \eqref{bc}, \eqref{ic} (see Definition \ref{def-admissibility}).

We shall separately comment on the case of the Dirichlet and Neumann boundary conditions in Section 3.2. By requiring that a solution belongs to the span of appropriate eigenvectors of the Laplace operator, it seems possible to formalize the latter situations as well. However, we find the construction slightly artificial. Therefore, we decided to go with the boundary conditions given in \eqref{bc}.

\subsection{Organization of the paper}
The paper is organized as follows. In Section 2, we recall notions and notations that we are going to use throughout the paper. In Section 3, we consider the mathematical analysis of equation  \eqref{che}, \eqref{ic}, \eqref{bc} where we use the Leray-Schauder fixed point theorem and Aubin-Lions lemma to prove the existence of the solution to the regularized problem. Then, we motivate and introduce the entropy solutions concept and prove the stability result. In Section 4 we provide the numerical procedure based on the convexity splitting method, and finally, in Section 5 we present the inpainting results of various binary shapes. We conclude the paper with final comments and a summary of the main results.

\section{Auxiliary notions and notations}

In this section, we shall introduce the notions and notations that we will use. In particular, we shall recall the Leray-Schauder fixed-point theorem and the Aubin-Lions lemma.

Constants $C(A)$ and $\tilde{C}(A)$ will denote positive generic constants depending on $A$. The operator $\nabla$ and $\Delta$ are always taken with respect to $\mx \in \Omega$ only. 

Let $\rho: \R\to \R_0^+$  be a smooth, compactly supported function, such that
\begin{align}\label{molif1}
\int_\R \rho(z) dz = 1,  \;\;\; \text{ and } \;\;\;
\rho_\eps(\mx) = \frac{1}{\eps^2}\rho\left(\frac{x_1}{\eps}\right)\rho\left(\frac{x_2}{\eps}\right), \ \ \mx=(x_1,x_2).
\end{align} We note that in the sense of distribution, it holds
$$
\rho_\eps \rightharpoonup \delta(x_1)\otimes \delta(x_2) \ \ {\rm as} \ \ \eps\to 0 \ \ {\rm in} \ \ {\cal D}(\R^2).
$$ For a measurable function $f$ defined on $[0, T]\times \Omega $, we define the standard convolution with respect to the spatial variable
\begin{align}
f_\eps(t,\mx):=f(t,\mx)\star_\mx \rho_\eps=\int_{\Omega} f(t,\my)\rho_\eps(\mx-\my)d\my, \ \ t\in [0, T], \; \mx\in \Omega.
\end{align} We stress that in the sequel, we shall distinguish between the $\mx$-convolution $f_\eps$ of a function $f$, and a function $f(t,\mx;\eps)$ depending on the parameter $\eps$. In particular, when we write $f_\eps(\cdot;\eps)$ we imply the function $f_\eps(t,\mx;\eps)=f(t,\mx;\eps)\star \rho_\eps$.    

%

In order to prove the existence of the solution we shall need the Leray-Schauder fixed point theorem and Lions-Aubins lemma \cite{Rou}. We recall them here for readers' convenience.
\begin{theorem} [Leray-Schauder fixed point theorem]
Let ${\cal T}$ be a continuous and compact mapping of a Banach space $X$ into itself, such that the set
$$
 \{x\in X: \, x=\sigma {\cal T}x \text{ for some }0\leq \sigma \leq 1\},
$$ is bounded. Then ${\cal T}$ has a fixed point.
\end{theorem}

\begin{theorem}[Aubin-Lions lemma]
Let $X_0$, $X$ and $X_1$ be three Banach spaces with $X_0\subset X \subset  X_1$. Suppose that $X_0$ is compactly embedded in $X$ and that $X$ is continuously embedded in $X_1$. For $1 \leq p, q  \leq \infty$, let
$$ 
W=\{u\in L^{p}([0,T];X_{0})| \, {\pa_t {u}}\in L^{q}([0,T];X_{1})\}.
$$

(i) If $p  < \infty$, then the embedding of $W$ into $L^p([0, T]; X)$ is compact.

(ii) If $p  = \infty$ and $q  >  1$, then the embedding of $W$ into $C([0, T]; X)$ is compact.

\end{theorem}


\section{Mathematical analysis of \eqref{che}, \eqref{ic}, \eqref{bc}}
In this section, we will address the mathematical aspects of existence and uniqueness of a weak solution (in an appropriate sense) of problem \eqref{che}, \eqref{ic}, \eqref{bc}. 
With the notation from the previous section, we start the analysis by considering a regularized equation of the following form
\begin{align}
\label{regularized}
\pa_t u(\,\cdot\,;\eps) &= \Delta \Big(-\nu \big( \arctan(\Delta u(\,\cdot \,;\eps) \star_\mx \rho_\eps)|\nabla u(\,\cdot \,;\eps)|\big)\star_\mx \rho_\eps - \mu \Delta u(\,\cdot\,;\eps) \Big) + \lambda(u_0 - u(\,\cdot\,;\eps))\\& :=\Delta \Big(-\nu \,\big(\arctan(\Delta u_\eps)|\nabla u|\big)_\eps - \mu \Delta u \Big) + \lambda(u_0 - u),
\nonumber
\end{align} where we denote $u=u(\,\cdot\,;\eps)$. We shall use the same notation in the sequel when there is no risk to make a confusion. We impose the same initial and boundary data as for \eqref{che} i.e. we take 
\begin{align}
\label{ic-bc}
 u(0,\mx;\eps)=u_0(\mx) \ \ {\rm and} \ \  u(\,\cdot\,;\eps) &= \frac{\partial u(\,\cdot\,;\eps)}{\partial \vec{n}} = 0   \;\; \text{on} \;\; [0,T]\times \partial\Omega.
\end{align}  

The plan is to prove an existence result for \eqref{regularized}, \eqref{ic-bc} and then to show that the family of solutions $(u(\,\cdot\,;\eps))_{\eps>0}$ to the regularized equations is uniformly bounded in $L^2([0,T];H^2(\Omega))\cap H^1([0,T];(H^{2}(\Omega))')$, $T>0$ (and thus strongly precompact in $L^2([0,T];H^1(\Omega))$ by the Aubin--Lions lemma).  Finally, by letting $\eps \to 0$, we introduce a solution concept in the frame of which we are able to prove well-posedness results for \eqref{che}, \eqref{ic}, \eqref{bc} as well.

We begin with the following theorem.

\begin{theorem}
\label{T1} For a fixed $\eps>0$, there exists a solution $u(\cdot;\eps)\in L^2([0,T];H^2(\Omega))$, $T>0$, to \eqref{regularized}, \eqref{ic-bc} whose bound in the $L^2([0,T];H^2(\Omega))$-norm is independent of $\eps$.
\end{theorem}
\begin{rem}
It is not difficult to see that the solution $u(\cdot;\eps)$ to \eqref{regularized}, \eqref{ic-bc} stated in the previous theorem is unique. However, this is not of particular importance for the paper as a whole.
\end{rem}
\begin{proof}
We are going to use a fairly standard fixed-point theorem technique. To this end, we define the mapping ${\cal T}: L^2([0,T];H^1(\Omega))\to L^2([0,T];H^1(\Omega))$ such that for a given $v\in L^2([0,T];H^1(\Omega))$, 
we take ${\cal T}(v)=u$ as the solution to the problem
 	\begin{equation} 
\label{regularized-LS}
\begin{split}
\pa_t u &= \Delta \Big(\big(-\nu\arctan(\Delta v_\eps)|\nabla v|\big)_\eps - \mu \Delta u \Big) + \lambda(u_0 - u),\\
u|_{t=0}&=u_0(\mx),  \ \ \ \ \ u|_{\pa \Omega}=  \frac{\pa u}{\pa \vec{n}}\big{|}_{\pa \Omega} =0.
\end{split}
\end{equation} The existence of a solution $u\in L^2([0,T];H^2(\Omega))$ to \eqref{regularized-LS} can be found in \cite[Section 7.6]{pazy} (see also \cite{BS, cher} for variants of the problem).
We shall prove that the mapping ${\cal T}$ satisfies the conditions of the Leray-Schauder theorem (see Section 2). First, we show that $u$ is bounded in $L^2([0,T];H^2(\Omega))$ by a constant which depends only on the $L^2([0,T];H^1(\Omega))$-norm of $v$ (up to an additive constant). Indeed, by a standard regularization argument (see the proof of Theorem \ref{T-uniqueness} for details), we can multiply \eqref{regularized-LS} by $u$, integrate over $[0,T] \times\Omega$ and take into account integration by parts and boundary conditions from \eqref{regularized-LS} to get:
\begin{align*}
\int_0^T\!\!\!\int_\Omega \frac{\pa_t u^2}{2} d\mx dt   &+   \mu \int_0^T\!\!\!\int_\Omega |\Delta u|^2  d\mx dt \\
&=  \int_0^T\!\!\!\int_\Omega \big(-\nu\arctan(\Delta v_\eps)|\nabla v|\big)_\eps \Delta u d\mx dt +  \int_0^T\!\!\!\int_\Omega\lambda(u_0 - u)ud\mx dt. 
\end{align*}  From here, using the initial conditions (left-hand side) and the discrete Young inequality (right-hand side),  we get 
\begin{align} \label{estimate0}
\int_\Omega \frac{u^2(T,\cdot)}{2} d\mx  &- \int_\Omega \frac{u_0^2(\mx)}{2} d\mx + \mu \int_0^T\!\!\int_\Omega |\Delta u|^2  d\mx dt
\\ \nonumber
&\leq  \int_0^T\!\!\int_\Omega  \left(\frac{\nu^2}{2\mu}\big(\arctan(\Delta v_\eps)|\nabla v|\big)_\eps^2+ \frac{\mu}{2}|\Delta u|^2\right) d\mx dt\\
&+  2\lambda_0\int_0^T\!\!\int_{\Omega\setminus \omega} u_0^2 d\mx dt - \frac{\lambda_0}{2}\int_0^T\!\!\int_{\Omega\setminus \omega} u^2 d\mx dt.  \nonumber
\end{align} Next, we use Young's convolution inequality for the term containing the arctan 
\begin{align}
\int_\Omega \big(\arctan(\Delta v_\eps)|\nabla v|\big)_\eps^2 d\mx   &\leq \int_\Omega \big(\arctan(\Delta v_\eps)|\nabla v|\big)^2d\mx  \;\left(\int_{\R^2} |\rho_\eps| d\mx\right)^2 \label{estConv}
\\&=\int_\Omega \big(\arctan(\Delta v_\eps)|\nabla v|\big)^2d\mx, \nonumber
\end{align} where we took into account $\|\rho_\eps\|_{L^1(\R^2)}^2 = 1$.
Furthermore, we can neglect the last term on the right-hand side of \eqref{estimate0} (since it is negative) and use \eqref{estConv} to obtain
\begin{align} 
\label{estimate6}
\int_\Omega \frac{u^2(T,\cdot)}{2} d\mx &- \int_\Omega \frac{u_0^2(\mx)}{2} d\mx  + \frac{\mu}{2} \int_0^T\!\!\int_\Omega |\Delta u|^2  d\mx dt  - {2\lambda_0} T \int_{\Omega\setminus \omega} u_0^2 d\mx  \\  
&  \leq \int_0^T\!\!\int_\Omega \frac{\nu^2}{2\mu}\big(\arctan(\Delta v_\eps)|\nabla v|\big)^2d\mx dt \nonumber \\  
& \leq  \int_0^T\!\!\int_\Omega  \tilde{C}|\nabla v|^2 d\mx dt\leq \tilde{C} \|v\|^2_{L^2([0,T];H^1(\Omega))}, \nonumber 
\end{align} 
since $\arctan(\Delta v_\eps)$ can be bounded by a constant, that we have included in $\tilde{C}$. We have thus proved that $\|{\cal T}(v)= u\|_{L^2([0,T];H^2(\Omega))}$ is bounded by  $ \|v\|_{L^2([0,T];H^1(\Omega))}+\|u_0\|_{L^2(\Omega)}$ (multiplied by a constant).

Using the Aubin-Lions lemma (given in Section 2 for readers' convenience), this is enough to conclude that the mapping ${\cal T}$ is compact. Indeed, we use the following Banach spaces:
\begin{align*}
X_0= H^2(\Omega), \ \ X=H^1(\Omega), \ \
X_1= H^{-2}(\Omega):=(H^{2}(\Omega))'. 
\end{align*}Clearly, $X_0$, $X$, and $X_1$ satisfy the conditions of the Aubin-Lions lemma. We now consider the set 
\begin{equation}
\label{W}
\begin{split}
{\cal T}(L^2([0,T];H^1(\Omega)))={\cal W}=\{u\in  L^2([0,T];H^1(\Omega)): \, \text{$u$ solves \eqref{regularized-LS} for some $v\in L^2([0,T];H^1(\Omega))$}\}.
\end{split}
\end{equation} In order to prove that ${\cal T}$ is compact, it is enough to show that ${\cal W}$ is compact with respect to the topology induced from $L^2([0,T];H^1(\Omega))$. Since the function $u$ solves \eqref{regularized-LS}, we have (keep in mind that $X_1=H^{-2}(\Omega)$)
\begin{equation}
\label{H1t}
\begin{split}
&\|\pa_t u\|_{L^2([0,T];X_1)}=\|\Delta \left(\nu \big(-\arctan(\Delta v_\eps)|\nabla v|\big)_\eps-\mu^2 \Delta u \right) + \lambda(u_0-u) \|_{L^2([0,T];X_1)}\\
&\leq \nu \| \Delta \left(\arctan(\Delta v_\eps)|\nabla v|\right)\|_{L^2([0,T];X_1)}+\mu\| \Delta^2 u \|_{L^2([0,T];X_1)} +  
\lambda_0\|(u_0-u) \|_{L^2([0,T] \times \Omega)}
\\& \leq \nu \|\arctan(\Delta v_\eps)|\nabla v|\|_{L^2([0,T] \times \Omega)}+\mu\| \Delta u  \|_{L^2([0,T]\times \Omega)}\\&\qquad + \lambda_0(\|u_0 \|_{L^2([0,T] \times \Omega)} + \|u \|_{L^2([0,T] \times \Omega)}) \leq \tilde{C}(\mu, \nu, \lambda_0, \|u_0\|_{L^2(\Omega)}, \|v\|_{L^2([0,T];H^{1}(\Omega))}). 
\end{split}
\end{equation} 
Thus, we have seen that ${\cal W}$ satisfies conditions of the Aubin-Lions lemma i.e. ${\cal W}$ is compactly embedded in $L^2([0,T];H^1(\Omega))$. 

From here, we see that the mapping ${\cal T}$ maps a bounded subset of $L^2([0,T];H^1(\Omega))$ to a compact subset of $L^2([0,T];H^1(\Omega))$ i.e. ${\cal T}$ is {a compact mapping}. It is also not hard to see that it is continuous as well (since solutions to the linear bi-parabolic equations are stable with respect to the coefficients). To conclude the arguments from the Leray-Schauder fixed point theorem, we need to prove that the set  
$$
B= \{u\in L^2([0,T];H^1(\Omega)): \, u= \sigma {\cal T}(u) \text{ for some }0\leq \sigma \leq 1\},
$$ is bounded. To this end, we go back to \eqref{estimate6} and discover from there that if $u=\sigma {\cal T}(u)$, then
\begin{align*}
\int_\Omega \frac{u^2(T,\cdot)}{2} d\mx dt- \int_\Omega \frac{u_0^2(\mx)}{2} d\mx  + \frac{\mu}{2} \int_0^T\!\!\!\int_\Omega |\Delta u|^2  d\mx dt &- {2\lambda_0}\int_0^T\!\!\int_{\Omega\setminus \omega} u_0^2 d\mx dt \\
&\leq \sigma C(\mu, \nu, \lambda_0) \|u\|^2_{L^2([0,T];H^1(\Omega))}.
\end{align*}
From the standard interpolation inequality (see e.g. \cite{GT}), we can find a constant $\tilde{C}(\mu, \nu, \lambda_0)$ such that (keep in mind that $0\leq \sigma \leq 1$)
\begin{align}
\label{estimate8}
\int_\Omega \frac{u^2(T,\cdot)}{2} d\mx dt - \int_\Omega \frac{u_0^2(\mx)}{2} d\mx &+  \frac{\mu}{2} \int_0^T\!\!\!\int_\Omega |\Delta u|^2  d\mx dt - 2{\lambda_0}\int_0^T\!\!\int_{\Omega\setminus \omega} u_0^2 d\mx dt \\
&\leq C(\mu, \nu, \lambda_0) \|u\|^2_{L^2([0,T];H^1(\Omega))} \nonumber \\
& \leq \tilde{C}(\mu, \nu, \lambda_0) \|u\|^2_{L^2([0,T]\times \Omega)}+\frac{\mu}{4} \int_0^T\!\!\!\int_\Omega |\Delta u|^2  d\mx dt. \nonumber
\end{align} From here, we get
\begin{equation*}
\int_\Omega \frac{u^2(T,\cdot)}{2} d\mx dt-(1+4\lambda_0)\int_\Omega \frac{u_0^2(\mx)}{2} d\mx+\frac{\mu}{4} \int_0^T\!\!\!\int_\Omega |\Delta u|^2  d\mx dt \leq \tilde{C}(\mu, \nu, \lambda_0) \|u\|^2_{L^2([0,T]\times \Omega)}.
\end{equation*} The Gronwall inequality implies
\begin{equation}
\label{bnd-app-sol}
\|u\|^2_{L^2([0,T]\times \Omega)}+\frac{\mu}{4} \|\Delta u\|_{L^2([0,T]\times \Omega)}^2   \leq (1+4\lambda_0) e^{\tilde{C}(\mu, \nu, \lambda_0)T}\int_\Omega \frac{u_0^2(\mx)}{2} d\mx, 
\end{equation} 
which together with \eqref{estimate8} and the interpolation inequality provides for constants $c, C>0$
\begin{equation*}
\| u\|_{L^2([0,T];H^1(\Omega))} \leq C\left(\int_0^T\!\! \int_\Omega u^2 d\mx dt +\int_0^T\!\!\int_\Omega |\Delta u|^2  d\mx dt\right) \leq c<\infty.
\end{equation*}  This proves boundedness of the set $B$. 
We have thus proved that all the conditions of the Leray-Schauder fixed point theorem are fulfilled, which implies the existence of the fixed point $u$ of the mapping ${\cal T}$. The fixed point $u$ is the solution to \eqref{regularized}, \eqref{ic-bc}. 
\end{proof} 

Next, we aim to let the regularization parameter $\eps$ tend to 0. We have the following theorem.

\begin{theorem}
\label{existence}
The family $(u(\,\cdot\,;\eps))$ converges as $\eps\to 0$ strongly in $L^2([0,T];H^1(\Omega))$ along a susbequence toward $u\in L^2([0,T];H^2(\Omega))$. 
\end{theorem}
\begin{proof}
Note that by \eqref{H1t} and \eqref{bnd-app-sol} we see that $(u(\,\cdot\,;\eps))$ is bounded in $L^2([0,T];H^2(\Omega))\cap H^1([0,T];H^{-2}(\Omega))$ independently of $\eps$. Using the Aubin-Lions lemma as in the proof of Theorem \ref{T1}, from here it follows that $(u(\,\cdot\,;\eps))$ is strongly precompact in $L^2([0,T];H^1(\Omega))$. Thus, there exists a zero--subsequence $(\eps_n)$ such that 
\begin{equation}
\label{H1conv}
u(\,\cdot \,;\eps_n)\to u \ \ {\rm in} \ \   L^2([0,T];H^1(\Omega)) \ \  {\rm as}  \ \ n\to \infty.
\end{equation} Note that we have in particular
\begin{equation*}
|\nabla u(\,\cdot \,;\eps_n)| \to |\nabla u| \ \ {\rm in} \ \ L^2([0,T]\times \Omega) \ \ {\rm as} \ \ n\to \infty.
\end{equation*} 
The fact $u\in L^2([0,T];H^2(\Omega))$ follows from \eqref{bnd-app-sol} (since the family $(u(\,\cdot\,;\eps))$ is bounded in $L^2([0,T];H^2(\Omega))$).
\end{proof}

A useful property of the function $u$ given in the previous theorem is the boundedness of $\pa_t u$ as a functional from the space 
\begin{equation}
\label{W2}
W_2=\{ g\in L^2([0,T];H^2(\Omega)): \, g(t,\mx) = \frac{\pa g}{\pa \vec{n}}(t,\mx) = 0 \text{ for } (t,\mx) \in [0,T]\times {\pa \Omega}\},
\end{equation} with the norm induced by $L^2([0,T];H^2(\Omega))$. More precisely, we have the following 
\begin{corollary}
\label{cor-1}
Let $u$ be the function given in \eqref{H1conv} and denote by $u^\epsilon$ its convolution with respect to $t\geq 0$:

$$
u^\epsilon(t,\mx)=u(\cdot,\mx)\star \frac{1}{\epsilon}\rho(t/\epsilon), \; \; \; \; \; \rm{supp}(\rho) \subset (-1,1). 
$$
The family $(u^\epsilon)$ satisfies for any $\phi\in W_2$
\begin{equation}
\label{t-der}
\big| \langle \pa_t u^\epsilon, \phi \rangle \big| \leq c \|\phi\|_{L^2([0,T];H^2(\Omega))}.
\end{equation} for a constant $c$ independent of $\eps$.
\end{corollary} 
\begin{proof} First note that since $(\arctan(\Delta u(\,\cdot\,;\eps)_\eps))$ is a family of bounded functions, there exists its subsequence (not relabled) such that

\begin{equation}
    \label{Linf}
    \arctan(\Delta u(\,\cdot\,;\eps)_\eps) \overset{\star}{\rightharpoonup} U \ \ {\rm as} \ \ \eps\to 0
\end{equation} weak-$\star$ in $L^\infty([0,T]\times \Omega)$ for a function $U\in L^\infty([0,T]\times \Omega)$. We note that $|U|$ is bounded by $\pi/2$ since $\arctan$ is bounded by $\pi/2$. Thus, by Theorem \ref{existence},  we have in a weak sense (with test functions from $W_2$; see \eqref{W2}) by letting $\eps\to 0$ in \eqref{regularized} along a subsequence 
\begin{equation}
\label{regularized-limit}
    \begin{split}
        \pa_t u &=\Delta \Big(-\nu \,U(t,\mx) |\nabla u| - \mu \Delta u \Big) + \lambda(u_0 - u),
    \end{split}
\end{equation} where we used \eqref{Linf} to handle the nonlinear term in \eqref{regularized}.

We fix $\tau >\epsilon$ and test \eqref{regularized-limit} against $\frac{1}{\epsilon}\rho((\tau - t)/\epsilon) \phi(\tau,\mx)$, $\phi \in W_2$, to get:
\begin{equation*}\begin{split}
    &\int_{\Omega}\pa_\tau\Big(\int_0^T u(t,\mx) \frac{1}{\epsilon}\rho((\tau - t)/\epsilon) dt\Big)\,  \phi(\tau,\mx)d\mx\\
    &=  \int_0^T \int_{\Omega} \Big(-\nu \,U(t,\mx) |\nabla u| - \mu \Delta u \Big) \frac{1}{\epsilon}\rho((\tau - t)/\epsilon) \Delta\phi(\tau, \mx)d\mx dt \\ 
    &+  \int_0^T \int_{\Omega} \lambda(u_0 - u)\frac{1}{\epsilon}\rho((\tau - t)/\epsilon) \phi(\tau, \mx)d\mx dt.
    \end{split}
\end{equation*}
Next, after integrating the latter over $\tau \in[0,T]$ and denoting by $F^\eps$ the convolution of $F$ with respect to $t$, we obtain
\begin{align*}
\big| \langle \pa_t u^\epsilon, \phi \rangle \big| &\leq \Big|\int_0^T\int_\Omega \left(\nu \,U(\tau,\mx) \,  |\nabla u|    + \mu \Delta u \right)^\epsilon \, \Delta \phi \, d\mx d\tau\Big|  + \Big|\int_0^T\int_\Omega \lambda(u_0 - u)^\epsilon \, \phi \, d\mx d\tau\Big|.
\end{align*} Here, we use the Cauchy-Schwarz inequality, Young inequality for convolution, and $|U|\leq \pi/2 $ to conclude
\begin{align*}
\big| \langle \pa_t u^\epsilon, \phi \rangle \big| &\leq \Big( c \|\nabla u\|_{L^2([0,T]\times \Omega)}    + \mu \|\Delta u\|_{L^2([0,T]\times \Omega))}\Big) \, \|\Delta \phi \|_{L^2([0,T]\times \Omega)}  \\& + \|\lambda(u_0 - u)\|_{L^2([0,T]\times \Omega)} \, \|\phi \|_{L^2([0,T]\times \Omega)} \leq C \|\phi\|_{{L^2([0,T];H^2(\Omega))}},
\end{align*} for a constant $C$ depending on  $\nu$, $\lambda_0$, and $L^2$-norms of $\Delta u$, $u$, and $u_0$. \end{proof}


%

We shall conclude the section by stating another property of the function $u$ given in Theorem \ref{existence} -- the existence of weak initial traces of $u$. We note that the next theorem is actually more general and states the weak trace property for any $u$ satisfying \eqref{regularized-limit}, provided that $U$ is bounded.

\begin{theorem}
\label{t-weak-traces} The initial condition $u_0$ is the weak trace of the function $u$ satisfying \eqref{regularized-limit} in the sense that for every $\phi \in W_2$ (see \eqref{W2}), it holds
\begin{equation}
\label{weak-trace}
{\rm ess}\!\lim\limits_{\!t\to 0} \int_{\Omega} u(t,\mx) \phi(\mx) d\mx = \int_{\Omega} u_0(\mx) \phi(\mx) d\mx.
\end{equation}
\end{theorem}
\begin{remark}
    We note that \eqref{weak-trace} holds also for $\phi\in L^2([0,T]\times \Omega)$ since $C^\infty_c([0,T]\times \Omega) \subset W_2$ is dense in $L^2([0,T]\times \Omega)$.
\end{remark}
\begin{proof}
We first test \eqref{regularized-limit} with respect to the function of the form
$$
\varphi(t,\mx)=\phi(\mx) \chi^\eps_{t_0}(t) 
$$ where $\phi \in W_2$
$$
\chi^\eps_{t_0}(t)=
\begin{cases}
0, & t \geq t_0\\
\frac{t_0-t}{\eps}, & t_0-\eps \leq t \leq t_0\\
1, & else
\end{cases}.
$$After taking into account that $\chi^\eps_{t_0}$ is supported in $[0,t_0]$, we have
\begin{align*}
&\frac{1}{\eps}\int_{t_0-\eps}^{t_0}\int_\Omega u\, \phi \, d\mx dt  -\int_\Omega u_0(\mx)\, \phi(\mx) d\mx \\& = -\int_0^{t_0}\int_\Omega \left(\nu \, U(t,\mx) \,  |\nabla u|    + \mu \Delta u \right)\chi^\eps_{t_0} \, \Delta \phi \, d\mx dt  + \int_0^{t_0}\!\int_\Omega \lambda(u_0 - u) \, \chi^\eps_{t_0} \, \phi \, d\mx dt.
\nonumber
\end{align*} Next, assume that $t_0\in E$ where $E$ is the set of Lebesgue points of the function $t \mapsto \int_\Omega u(t,\mx) \phi(\mx) \, d\mx$ and let $\eps\to 0$ in the latter expression. We note that the set $E$ is of the full measure. We get after taking into account $\chi^\eps_{t_0}(t)\overset{\eps\to 0}{\longrightarrow} {\bf 1}_{[0,t_0]}$ almost everywhere (where ${\bf 1}_{[0,t_0]}$ is the characteristic function of the set $[0,t_0]$) 
\begin{align*}
&\int_\Omega u(t_0,\mx) \phi(\mx) \, d\mx dt  -\int_\Omega u_0(\mx)\, \phi(\mx) d\mx \\& = -\int_0^{t_0}\int_\Omega \left(\nu \,U(t,\mx) \,  |\nabla u|    + \mu \Delta u \right) \, \Delta \phi \, d\mx dt  + \int_0^{t_0}\!\int_\Omega \lambda(u_0 - u) \,  \, \phi \, d\mx dt
\nonumber
\end{align*} To conclude \eqref{weak-trace}, it is now left to let $t_0\to 0$ along the above defined (full) set of the Lebesgue points $E$, and to take into account $u\in L^2([0,T];H^2(\Omega))$ and boundedness of the function $U$. The proof is concluded. \end{proof}

\subsection{Entropy solution concept and uniqueness}

As we have seen in the previous section, the sequence of approximate solutions $(u(\,\cdot\,;\eps))$ converges strongly in $L^2([0,T];H^1(\Omega))$ towards a function $u\in L^2([0,T];H^2(\Omega))$. By letting $\eps\to 0$ in \eqref{regularized} along an appropriate subsequence, we see that $u$ satisfies \eqref{regularized-limit}, but it is not clear which form takes the function $U$. Lack of this information prevents us from proving the uniqueness of the function $u$. We note that more on the form of $U$ can be said using the Young measures framework (see Subsection 3.2), but it does not substantially affect the well-posedness result.  

Therefore, we are going to look for the properties of equation \eqref{regularized} which allow us to avoid the nonlinear term $\arctan(\Delta u_\eps(\,\cdot\,;\eps))|\nabla u|$. We will end up with an inequality that is usually called entropy admissibility conditions (see \cite{Kru} where such an approach was essentially initiated). To motivate them, we start by testing the regularized equation by $u-\varphi$ where $\varphi \in C^\infty([0,T]\times \Omega)$ is such that it satisfies \eqref{ic} and \eqref{bc}:
\begin{equation}
\label{cond-fi}
\varphi(t,\mx)=\frac{\pa \varphi(t,\mx) }{\pa  \vec{n}}|_{\pa \Omega}=0 \ \ {\rm for} \ \ (t,\mx)\in [0,T]\times \pa \Omega.
\end{equation} We have for every $t\in[0,T]$ (employing the approximation argument as in Corollary \ref{cor-1} if necessary)
\begin{equation} \label{solu}
\begin{split}
&\int_0^t \int_{\Omega}\left( u-\varphi \right) \pa_t u \, d\mx dt'  \\&= -\int_0^t\int_\Omega  \Big(\nu \arctan(\Delta u_\eps)|\nabla u|\, \Delta(u-\varphi)_\eps+ \mu \Delta u \, \Delta(u-\varphi) \Big)\,d\mx dt'  \\& \ \ \;+ \int_0^t\int_\Omega \lambda(u_0 - u)(u-\varphi)d\mx dt'.
\end{split}
\end{equation} 
Next, we add and subtract $\arctan(\Delta \varphi_\eps)$ in the first summand on the right-hand side and $\Delta \varphi$ in the second summand on the right-hand side, and use the fact that $\arctan$ is an increasing function. We have

\begin{align} \label{solu-1}
&\int_0^t \int_{\Omega}\frac{1}{2}\pa_t\left( u-\varphi \right)^2   d\mx dt'+\int_0^t \int_{\Omega}(u-\varphi)\pa_t \varphi d\mx dt'  \\&= -\int_0^t\int_\Omega  \Big(\nu (\arctan(\Delta u_\eps)-\arctan(\Delta \varphi_\eps))|\nabla u| \Delta(u-\varphi)_\eps \nonumber\\&\qquad\qquad\qquad +\nu \arctan(\Delta \varphi_\eps)|\nabla u| \Delta(u-\varphi)_\eps +\mu \Delta (u-\varphi) \Delta(u-\varphi) +\mu \Delta \varphi \Delta(u-\varphi)  \Big)d\mx dt' \nonumber \\&\quad + \int_0^t\int_\Omega \lambda(u_0 - u)(u-\varphi)d\mx dt'
\nonumber \\&\leq -\int_0^t\int_\Omega  \Big( \nu \arctan(\Delta \varphi_\eps)|\nabla u|\, \Delta(u-\varphi)_\eps +\mu \Delta \varphi \Delta(u-\varphi) \Big)d\mx dt' \nonumber \\&\quad + \int_0^t\int_\Omega \lambda(u_0 - u)(u-\varphi)d\mx dt'.
\nonumber 
\end{align} Now, we want to let $\eps\to 0$. We recall that along a (non-relabeled) subsequence we have (keep in mind that in \eqref{solu-1} we have $u:=u(\,\cdot\,;\eps)$ and the function $u$ below is from Theorem \ref{existence} i.e. it is independent of $\eps$)
\begin{equation}
\label{conv}
\begin{split}
u(\,\cdot\,;\eps) \to u \ \ \text{strongly in} \ \ L^2([0,T]\times \Omega);\\
\nabla_{\mx} u(\,\cdot\,;\eps) \to \nabla_{\mx} u \ \ \text{strongly in} \ \ L^2([0,T]\times \Omega);\\
\Delta u(\,\cdot\,;\eps) \rightharpoonup \Delta u \ \ \text{weakly in} \ \ L^2([0,T]\times \Omega);\\
\varphi_\eps \to \varphi \ \ \text{strongly in} \ \ C([0,T]\times \Omega).
\end{split}
\end{equation} Thus, letting $\eps\to 0$ in \eqref{solu-1} along the subsequence chosen in Theorem \ref{existence}, we reach to the following entropy solutions concept.

\begin{definition}
\label{def-admissibility} We say that the function $u\in L^2([0,T];H^2(\Omega))$ represents an entropy solution to problem \eqref{che}, \eqref{ic}, \eqref{bc} if 
 the following inequality holds for almost every $t\in [0,T]$ and every $\varphi \in Lip([0,T];H^2(\Omega))$ satisfying \eqref{cond-fi}:
\begin{equation} \label{solu-2}
\begin{split}
& \int_{\Omega}\frac{1}{2}( u(t,\mx)-\varphi(t,\mx))^2   d\mx -\int_{\Omega}\frac{1}{2}( u_0(\mx)-\varphi(0,\mx))^2   d\mx  + \int_0^t \!\!\int_{\Omega}(u-\varphi)\pa_t \varphi d\mx dt' \\& \leq -\int_0^t\!\!\int_\Omega  \Big(\nu \arctan(\Delta \varphi)|\nabla u| \Delta(u-\varphi)  +\mu \Delta \varphi \Delta(u-\varphi) \Big)d\mx dt'  \\
&+ \int_0^t\int_\Omega \lambda(u_0 - u)(u-\varphi)d\mx dt'. 
\end{split}
\end{equation}

\end{definition}

We note that a standard density argument enables us to take $\varphi \in Lip([0,T];H^2(\Omega))$ instead of $\varphi \in C^\infty([0,T]\times \Omega)$.

Although given in the form of an inequality (see Subsection 3.2 for further comments), the previous definition provides stability of solutions and essentially singles out solutions obtained by the limiting procedure from the proof of Theorem \ref{existence}. Before we prove the well-posedness of entropy solutions to \eqref{che}, \eqref{ic}, \eqref{bc}, we need to prove that such solutions admit strong traces at $t=0$.
\begin{theorem}
An entropy solution to \eqref{che}, \eqref{ic}, \eqref{bc} in the sense of Definition \ref{def-admissibility} admits the strong trace at $t=0$ in the sense
\begin{equation}
\label{strong-trace}
{\rm ess}\!\lim\limits_{t\to 0} \int_\Omega |u(t,\mx)-u_0(\mx)|^2 d\mx=0.
\end{equation}
\end{theorem}
\begin{proof}
We follow the proof of Theorem \ref{t-weak-traces} and, with the notation from there, we take (for an $\eps$ such that $t_0-\eps>0$) 
$$
\varphi(t,\mx)=\chi^\eps_{t_0}(t)u_0(\mx).
$$ We get from \eqref{solu-2} after taking $t=t_0-\eps$:
\begin{align*} 
& \int_{\Omega}\frac{1}{2}( u(t_0-\eps,\mx)-u_0(\mx))^2   d\mx  
\\& \leq -\int_0^{t_0-\eps}\!\!\int_\Omega  \Big(\nu \arctan(\chi^\eps_{t_0} \Delta u_0)|\nabla u| \Delta(u-\chi^\eps_{t_0}  u_0) \nonumber +\mu \chi^\eps_{t_0} \,\Delta u_0\, \Delta(u-\chi^\eps_{t_0}u_0) \Big)d\mx dt' \nonumber \\&+ \int_0^{t_0-\eps}\int_\Omega \lambda(u_0 - u)(u-\chi^\eps_{t_0} u_0)d\mx dt'. 
\nonumber
\end{align*} By letting $\eps\to 0$ along a subsequence defining the value of $ t\mapsto \int_{\Omega}\frac{1}{2}(u(t,\mx)-u_0(\mx))^2   d\mx $ in its Lebesgue point $t_0$ here (we always tacitly assuming that $t_0$ is the Lebesgue point of an appropriate function), we have
\begin{align*} 
& \int_{\Omega}\frac{1}{2}( u(t_0,\mx)-u_0(\mx))^2   d\mx   
\\& \leq -\int_0^{t_0}\!\!\int_\Omega  \Big(\nu \arctan(\Delta u_0)|\nabla u| \Delta(u- u_0) \nonumber +\mu  \,\Delta u_0\, \Delta(u-u_0) \Big)d\mx dt' \nonumber \\&+ \int_0^{t_0}\int_\Omega \lambda(u_0 - u)(u- u_0)d\mx dt'. 
\nonumber
\end{align*} Finally, letting here $t_0\to 0$, we reach to the conclusion of the theorem. \end{proof}

\begin{theorem}
\label{T-uniqueness}
There exists a unique entropy solution to \eqref{che}, \eqref{ic}, \eqref{bc} in the sense of Definition \ref{def-admissibility}.
\end{theorem} 
\begin{proof}
As we have seen from \eqref{solu}--\eqref{conv}, the family of solution $(u(\cdot, \eps))$ to \eqref{regularized-LS} strongly converges along a subsequence as $\eps\to 0$ in $L^2([0,T];H^1(\Omega))$ toward an entropy solution $u(t,\mx)$ to \eqref{che}, \eqref{ic}, \eqref{bc} in the sense of Definition \ref{def-admissibility}.

Now, we want to prove the stability of entropy solutions to \eqref{che}, \eqref{ic}, \eqref{bc}. We take two entropy solutions to the latter problem say $u$ and $v$ with the initial data $u_0$ and $v_0$, respectively. We denote by $u^\epsilon$ and $v^\epsilon$ mollification of $u$ and $v$ with respect to $t$:
$$
u^\epsilon(t,\mx)=u(\cdot,\mx)\star \frac{1}{\epsilon}\rho(t/\epsilon), \ \ v^\epsilon(t,\mx)=v(\cdot,\mx)\star \frac{1}{\epsilon}\rho(t/\epsilon). 
$$ Now, write \eqref{solu-2} for $u$ and $v$ and write $v^\epsilon$ and $u^\epsilon$ instead of $\varphi$, respectively. We have
\begin{equation} 
\label{solu-2-u}
\begin{split}
& \int_{\Omega}\frac{1}{2}( u(t,\mx)-v^\epsilon(t,\mx))^2   d\mx -\int_{\Omega}\frac{1}{2}( u_0(\mx)-v^\epsilon(0,\mx))^2   d\mx  +\! \int_0^t \!\!\int_{\Omega}(u-v^\epsilon)\pa_t v^\epsilon d\mx dt' \\&  \leq \!-\int_0^t\!\!\int_\Omega \! \Big(\nu \arctan(\Delta v^\epsilon)|\nabla u| \, \Delta(u-v^\epsilon)  \!+\!\mu \Delta v^\epsilon \Delta(u-v^\epsilon) \Big)d\mx dt'  \\& \; + \int_0^t\int_\Omega \lambda\,(u_0 - u)(u-v^\epsilon)d\mx dt', 
\end{split}
\end{equation} and
\begin{align} \label{solu-2-v}
& \int_{\Omega}\frac{1}{2}( v(t,\mx)-u^\epsilon(t,\mx))^2   d\mx -\int_{\Omega}\frac{1}{2}( v_0(\mx)-u^\epsilon(0,\mx))^2   d\mx  + \!\int_0^t \!\!\int_{\Omega}(v-u^\epsilon)\pa_t u^\epsilon d\mx dt' \\& \leq\!-\int_0^t\!\!\int_\Omega \! \Big(\nu \arctan(\Delta u^\epsilon)|\nabla v|\, \Delta(v-u^\epsilon) \nonumber \!+\!\mu \Delta u^\epsilon \Delta(v-u^\epsilon) \Big)d\mx dt' \nonumber \\&+ \int_0^t\int_\Omega \lambda(v_0 - v)(v-u^\epsilon)d\mx dt'.
\nonumber
\end{align} Next, we add \eqref{solu-2-u} and \eqref{solu-2-v} and we aim to let $\epsilon\to 0$. We have
\begin{align} \label{solu-3-uv}
& \int_{\Omega}\frac{1}{2}\left( (u(t,\mx)-v^\epsilon(t,\mx))^2 +( v(t,\mx)-u^\epsilon(t,\mx))^2 \right) d\mx \\&
-\int_{\Omega}\frac{1}{2}( u_0(\mx)-v^\epsilon(0,\mx))^2   d\mx  -\int_{\Omega}\frac{1}{2}( v_0(\mx)-u^\epsilon(0,\mx))^2 d\mx \nonumber\\
&+ \int_0^t \!\!\int_{\Omega}(u-v^\epsilon)\pa_t v^\epsilon d\mx dt' + \int_0^t \!\!\int_{\Omega}(v-u^\epsilon)\pa_t u^\eps d\mx dt' \nonumber \\
&\leq -\int_0^t\!\!\int_\Omega  \nu \Big(\arctan(\Delta v^\epsilon)|\nabla u|\, \Delta(u-v^\epsilon)+\arctan(\Delta u^\epsilon)|\nabla v|\, \Delta(v-u^\epsilon) \Big)d\mx dt' \nonumber\\& 
-\mu \int_0^t\!\!\int_\Omega \left(\Delta v^\epsilon \Delta(u-v^\epsilon)+\Delta u^\epsilon \Delta(v-u^\epsilon) \right) d\mx dt' \nonumber \\&
+ \int_0^t\int_\Omega \lambda(u_0 - u)(u-v^\epsilon)\,d\mx dt'+\int_0^t\int_\Omega \lambda(v_0 - v)(v-u^\epsilon)\,d\mx dt', 
\nonumber
\end{align} Note that by \eqref{strong-trace}, we have in $L^2(\Omega)$ as $\epsilon\to 0$
\begin{equation}
\label{stron-trace-1}
v^\epsilon(0,\cdot) \to v_0 \ \ {\rm and} \ \ u^\epsilon(0,\cdot) \to u_0.
\end{equation} Let us consider the last two terms on the left-hand side of the previous expression. We have
\begin{equation}
\label{U1}
\begin{split}
&\int_0^t \!\!\int_{\Omega}(u-v^\epsilon)\pa_t v^\epsilon d\mx dt' + \int_0^t \!\!\int_{\Omega}(v-u^\epsilon)\pa_t u^\epsilon d\mx dt'
\\&=\int_0^t \!\!\int_{\Omega}(u^\epsilon-v^\epsilon)\pa_t v^\epsilon d\mx dt' + \int_0^t \!\!\int_{\Omega}(v^\epsilon-u^\epsilon)\pa_t u^\epsilon d\mx dt'
\\&+\int_0^t \!\!\int_{\Omega}(u-u^\epsilon)\pa_t v^\epsilon d\mx dt' + \int_0^t \!\!\int_{\Omega}(v-v^\epsilon)\pa_t u^\epsilon d\mx dt'
\\&=-\frac{1}{2}\int_0^t \!\!\int_{\Omega}\pa_t (v^\epsilon-u^\epsilon)^2 d\mx dt' +o_\epsilon(1),
\end{split}
\end{equation} where $o_\epsilon(1)$ is as before the Landau symbol denoting the function tending to zero as $\epsilon\to 0$. Indeed, by Corollary \ref{cor-1}, we have 
\begin{align*}
|o_\epsilon(1)|:&=\Big|\int_0^t \!\!\int_{\Omega}(u-u^\epsilon)\pa_t v^\epsilon d\mx dt' + \int_0^t \!\!\int_{\Omega}(v-v^\epsilon)\pa_t u^\epsilon d\mx dt'\Big|
\\&\leq c \left( \|u-u^\epsilon\|_{L^2([0,T];H^2(\Omega))}+\|v-v^\epsilon\|_{L^2([0,T];H^2(\Omega))} \right)\to 0 \ \ {\rm as} \ \ \epsilon\to 0
\end{align*} since $u^\epsilon$ and $v^\epsilon$ are merely convolutions of $L^2([0,T];H^2(\Omega))$ functions $u$ and $v$, respectively.
Taking this, \eqref{U1}, and \eqref{stron-trace-1} into account, we get after letting $\epsilon\to 0$ in \eqref{solu-3-uv}

\begin{align*} 
& \int_{\Omega}\frac{1}{2} (u(t,\mx)-v(t,\mx))^2  d\mx 
-\int_{\Omega}\frac{1}{2}( u_0(\mx)-v_0(\mx))^2   d\mx  \\
&\leq -\int_0^t\!\!\int_\Omega  \nu \Big(\arctan(\Delta v)|\nabla u|\, \Delta(u-v)+\arctan(\Delta u)|\nabla v|\, \Delta(v-u) \Big)d\mx dt' \nonumber\\& \ \  
-\mu \int_0^t\!\!\int_\Omega \left(\Delta v \, \Delta(u-v)+\Delta u \, \Delta(v-u) \right) d\mx dt' \nonumber \\& \ \ 
+ \int_0^t\int_\Omega \lambda(u_0 - u)(u-v)\,d\mx dt'+\int_0^t\int_\Omega \lambda(v_0 - v)(v-u)\,d\mx dt', 
\nonumber
\end{align*} Rearranging the $\arctan$ terms and using the fact that $\arctan$ is an increasing function similarly as in \eqref{solu-1}, we have from the latter equation 

\begin{align*} 
& \int_{\Omega}\frac{1}{2} (u(t,\mx)-v(t,\mx))^2  d\mx 
-\int_{\Omega}\frac{1}{2}( u_0(\mx)-v_0(\mx))^2   d\mx  \\
&\leq -\int_0^t\!\!\int_\Omega  \nu \big(\arctan(\Delta u)-\arctan(\Delta v) \big)|\nabla u|\, \Delta(u-v) d\mx dt' \nonumber\\&
 \ \ +\int_0^t\!\!\int_\Omega  \nu \arctan(\Delta u)(|\nabla u|-|\nabla v|)\, \Delta(v-u) \, d\mx dt'
\nonumber \\& \ \  
-\mu \int_0^t\!\!\int_\Omega \left(\Delta(u-v)\right)^2 d\mx dt'  
+ \int_0^t\int_\Omega \lambda(u_0 - u-v_0+u)(u-v)\,d\mx dt'
\\& \leq \int_0^t\!\!\int_\Omega  \nu \arctan(\Delta u)\,|\nabla (u- v)| \, \Delta(v-u) \, d\mx dt'
\nonumber \\& \ \  
-\mu \int_0^t\!\!\int_\Omega \left(\Delta(u-v)\right)^2 d\mx dt'  
+ \int_0^t\int_\Omega \lambda(u_0 - u-v_0+u)(u-v)\,d\mx dt'
\nonumber
\end{align*} From the above, we have upon using the Cauchy-Schwartz and discrete Young inequality 
\begin{align*} 
& \int_{\Omega}\frac{1}{2} (u(t,\mx)-v(t,\mx))^2  d\mx +\mu \int_0^t\!\!\int_\Omega \left(\Delta(u-v)\right)^2 d\mx dt'
\\& \leq \int_{\Omega}( u_0-v_0)^2\,d\mx+ C(\mu)\int_0^t\!\!\int_\Omega |\nabla (u- v)|^2 \, d\mx dt'+ 
\frac{\mu}{2}\int_0^t\!\!\int_\Omega \big| \Delta(v-u) \big|^2 \, d\mx dt'
\nonumber \\&   
+ \int_0^t\int_\Omega \lambda(u_0-v_0)^2\,d\mx dt +2\int_0^t\int_\Omega(u-v)^2\,d\mx dt'
\nonumber
\end{align*} for a constant $C(\mu)>0$. After simple rearrangements and by the interpolation inequality as in \eqref{estimate8} (applied to $\int_0^t\!\!\int_\Omega |\nabla (u- v)|^2 \, d\mx dt'$) this becomes
\begin{align*} 
& \int_{\Omega}\frac{1}{2} (u(t,\mx)-v(t,\mx))^2  d\mx +\frac{\mu}{4} \int_0^t\!\!\int_\Omega \left(\Delta(u-v)\right)^2 d\mx dt'
\\& \leq (1+\lambda_0 t)\int_{\Omega}( u_0-v_0)^2\,d\mx+ \tilde{C}(\mu)\int_0^t\!\!\int_\Omega ( u- v)^2 \, d\mx dt' 
\nonumber
\end{align*}for another constant $\tilde{C}(\mu)$. From the Gronwall inequality, analogically with \eqref{bnd-app-sol}, we see that it holds
$$
\int_{\Omega}\frac{1}{2} (u(t,\mx)-v(t,\mx))^2  d\mx \leq \bar{C} \int_{\Omega}( u_0-v_0)^2\,d\mx 
$$ again for a constant $\bar{C}$ depending on $\lambda_0$, $t$, and $\tilde{C}(\mu)$. This concludes the stability proof. \end{proof}

\subsection{Solution concept as a variational equality}
Following the internal logic behind entropy solutions \cite{Kru} or viscosity solutions \cite{Lio}, it is expected that if $u \in C^4([0,T] \times \Omega)$ satisfies \eqref{solu-2}, then it should satisfy \eqref{che-old} in the classical sense (the vice versa is clear). This seems not retrievable from \eqref{solu-2}, and it is a natural question what kind of variational equality $u$ satisfies.

At this moment, the most accurate answer can be obtained via the variant of Young measures given in \cite{DiPM}. To recall them, we denote by ${\cal M}([0,T] \times \Omega)$ the space of Radon measures on $[0,T] \times \Omega$ that have a finite total mass and by ${\cal M}^+(\mathbb{R})$ the space of non-negative Radon measures on $\mathbb{R}$. We have the following adaptation of the theorem on generalized Young measures.


\begin{theorem} \cite[Theorem 1 \& Remark 1]{DiPM}
\label{young}
(Generalized Young Measures)  If $(u_\eps)$ is an arbitrary family of functions whose $L^2$--norm on a set $[0, T] \times \Omega \subset\subset \R^n$ is uniformly bounded, then $(u_\eps)$ contains a subsequence with the following properties. There exists a Radon measure $\meta \in {\cal M}([0, T] \times \Omega)$
$$ |u_\eps |^2 \rightharpoonup \meta   \text{ in }  {\cal M}([0, T] \times \Omega),$$  
and a $\meta-$ measurable map 
$$(t,\mx) \mapsto \pi_{(t,\mx)}^1$$
from $[0,T]\times \Omega$ to ${\cal M}^+(\R)$ such that for all $g$ of the following form 
$$g(v) = g_0(v)(1+|v|^2),  \ \ g_0\in C_0(\R)$$ 
we have 
$$g(u_\eps) \rightharpoonup  \langle \pi^1, g_0\rangle(1 + h)d\mx dt ,$$
where $h$ is the coefficient in the Lebesgue decomposition $\meta=\meta_s+h(t,\mx)dt d\mx$ with $\meta_s$ being a singular part, i.e.
\begin{align}
\label{220}
\lim_{\eps \to 0}\int_0^T\!\!\int_\Omega \phi g(u_\eps)d\mx dt = \int_0^T\!\!\int_\Omega \phi \langle \pi_{(t,x)}^1, g_0\rangle(1 + h)d\mx dt,
\end{align}
for all $\phi \in C_0^\infty([0,T]\times\Omega)$.
\end{theorem}

The previous theorem enables us to specify the limit in the nonlinear term from \eqref{regularized}. We have the following theorem which includes the variational equality defining a solution to \eqref{che}.

\begin{theorem}
    The function $u$ given in Theorem \ref{existence} satisfying the following variational equality for some Young measure $\pi_{(t,\mx)}\in L^1_w([0,T]\times \Omega;{\cal M}(\R))$ and
 for all $\varphi \in W_2\cap Lip([0,T];H^2(\Omega))$ (see \eqref{W2}) with $\varphi(T,\mx)=0$ for all $\mx\in \Omega$, we have

\begin{align}
\label{def-sol}
&\int_0^T\int_\Omega u\, \pa_t\varphi \, d\mx dt  +\int_\Omega u_0(\mx)\, \varphi(0,\mx) d\mx \\& = \int_0^T\int_\Omega \left(\nu \,\int_{\R}\arctan(\xi)d\pi_{(t,x)}(\xi)\,  |\nabla u|    + \mu \Delta u \right)\Delta \varphi \, d\mx dt  - \int_0^T\!\int_\Omega \lambda(u_0 - u) \, \varphi \, d\mx dt.
\nonumber
\end{align} 
\item 

The connection between $\pi_{(t,\mx)}(\xi)$ and $u$ is given by
\begin{equation}
    \label{connection}
\Delta u(t,\mx)=\int_{\R} \xi \, d\pi_{(t,\mx)}(\xi).
\end{equation}

Moreover, if $\pi_{(t,\mx)}(\xi)=\delta(\xi-\Delta u(t,\mx))$ then the function $u$ satisfies \eqref{che-old} in the standard sense of distributions on $(0,T)\times \Omega$.
\end{theorem}
\begin{proof}

The only suspicious term in \eqref{def-sol} is the one involving $\arctan$. To this end, we have in the weak sense i.e. for any fixed $\varphi  \in W_2 \cap Lip([0,T];H^2(\Omega))$ 
\begin{equation}
\label{susp-1}
\begin{split} 
&\langle -\Delta \nu \big( \arctan(\Delta u(\,\cdot \,;\eps_n) \star_\mx \rho_{\eps_n})|\nabla u(\,\cdot \,;\eps_n)|\big)\star_\mx \rho_{\eps_n}, \varphi\rangle\\&=
-\int_0^T\int_\Omega \nu \big( \arctan(\Delta u(\,\cdot \,;\eps_n) \star_\mx \rho_{\eps_n})|\nabla u(\,\cdot \,;\eps_n)|\big)\star_\mx \rho_{\eps_n} \, \Delta \varphi \, d\mx dt
\\& =-\int_0^T\int_\Omega \nu  \arctan(\Delta u(\,\cdot \,;\eps_n) \star_\mx \rho_{\eps_n}) \, |\nabla u(\,\cdot \,;\eps_n)| \, \Delta (\varphi\star_\mx \rho_{\eps_n}) \, d\mx dt.
\end{split}
\end{equation} Here, we have by \eqref{H1conv}
\begin{equation*}
|\nabla u(\,\cdot \,;\eps_n)| \to |\nabla u| \ \ {\rm in} \ \ L^2([0,T]\times \Omega) \ \ {\rm as} \ \ n\to \infty,
\end{equation*} and by Theorem \ref{young} (along a non-relabelled sequence) for any $\phi\in C^\infty_0([0,T]\times \Omega)$
\begin{equation}
\label{atan}
\int_0^T\int_\Omega \arctan(\Delta u(\,\cdot \,;\eps_n)) \, \phi \, d\mx dt \to \int_0^T\int_\Omega \phi \, \langle \pi_{(t,x)}^1,\frac{\arctan(\,\cdot\,)}{1+|\,\cdot\,|^2} \rangle(1 + h)d\mx dt
\end{equation} where, with the notations from the latter theorem,  we take $g_0(v) = \frac{\arctan(v)}{1+v^2}$ in \eqref{220}. If we write
$$
\pi_{(t,\mx)}(\xi)= \frac{1 + h(t,\mx)}{1+|\xi|^2}\pi_{(t,x)}^1(\xi), \ \ \xi\in \R,
$$ we reach to the conclusion of the theorem. Indeed, to show \eqref{def-sol}, we fix $\eta>0$ and denote $\varphi_\eta=\varphi \star_{t,\mx} \rho_\eta$ (we have taken the convolution with respect to {both variables}). We have from \eqref{susp-1} (after taking Theorem \ref{young} into account and since $\varphi_\eta$ is smooth):
\begin{equation}
\label{susp-2}
\begin{split} 
&\lim\limits_{n\to \infty}
\langle -\nu \Delta  \big( \arctan(\Delta u(\,\cdot \,;\eps_n) \star_\mx \rho_{\eps_n})|\nabla u(\,\cdot \,;\eps_n)|\big)\star_\mx \rho_{\eps_n}, \varphi\rangle\\&=
-\lim\limits_{n\to \infty}\int_0^T\int_\Omega \nu  \arctan(\Delta u(\,\cdot \,;\eps_n) \star_\mx \rho_{\eps_n}) \, |\nabla u(\,\cdot \,;\eps_n)| \, \Delta \varphi_\eta\star_\mx \rho_{\eps_n} \, d\mx dt\\
&-\lim\limits_{n\to \infty} \int_0^T\int_\Omega \nu  \arctan(\Delta u(\,\cdot \,;\eps_n) \star_\mx \rho_{\eps_n}) \, |\nabla u(\,\cdot \,;\eps_n)| \, \Delta (\varphi-\varphi_\eta)\star_\mx \rho_{\eps_n} \, d\mx dt
\\&=-\int_0^T\int_\Omega \nu \Big( \int_{\R}\arctan(\xi)d\pi_{(t,x)}(\xi)\Big)  |\nabla u| \Delta \varphi_\eta \, d\mx dt +o_\eta(1).
\end{split}
\end{equation} where $o_\eta(1)$ is a standard Landau symbol denoting the function tending to zero as $\eta\to 0$.

Indeed, it holds

\begin{equation*}
\begin{split}
&\lim\limits_{n\to \infty}\Big|\int_0^T\int_\Omega \nu  \arctan(\Delta u(\,\cdot \,;\eps_n) \star_\mx \rho_{\eps_n}) \, |\nabla u(\,\cdot \,;\eps_n)| \, \Delta (\varphi-\varphi_\eta)\star_\mx \rho_{\eps_n} \, d\mx dt\Big|
\\& \leq \lim\limits_{n\to \infty} \|\nu  \arctan(\Delta u(\,\cdot \,;\eps_n) \star_\mx \rho_{\eps_n}) \, |\nabla u(\,\cdot \,;\eps_n)|\|_{L^2([0,T]\times \Omega)}  \, \| \Delta (\varphi-\varphi_\eta) \|_{L^2([0,T]\times \Omega)} 
\end{split}
\end{equation*} and since $\| \Delta (\varphi-\varphi_\eta) \|_{L^2([0,T]\times \Omega)}\to 0$ as $\eta\to 0$, we see that \eqref{susp-2} indeed holds. Finally, letting $\eta\to 0$ in \eqref{susp-2}, we discover
\begin{equation}
\label{susp-3}
\begin{split}
&\lim\limits_{n\to \infty}
\langle -\nu \,\Delta  \big( \arctan(\Delta u(\,\cdot \,;\eps_n) \star_\mx \rho_{\eps_n})|\nabla u(\,\cdot \,;\eps_n)|\big)\star_\mx \rho_{\eps_n}, \varphi\rangle\\
&=-\int_0^T\int_\Omega \nu \Big( \int_{\R}\arctan(\xi)d\pi_{(t,x)}(\xi)\Big)  |\nabla u| \Delta \varphi \, d\mx dt.
\end{split}
\end{equation} From the above, we see that we can test \eqref{regularized} by an arbitrary $\varphi \in W_2\cap Lip([0,T];H^2(\Omega))$, take into account conditions \eqref{ic-bc}, and let $\eps\to 0$ along the subsequence $(\eps_n)$ given above to reach to \eqref{def-sol}.

To prove the last part of the theorem, we simply notice that
$$
\Delta u(\cdot;\eps_n) \rightharpoonup \Delta u \ \ {\rm in} \ \ L^2([0,T]\times \Omega) \ \ {\rm as} \ \ n\to \infty.
$$ On the other hand, we can express the limit of $(\Delta u(\cdot;\eps_n))$ via the measure $d \pi_{(t,\mx)}(\xi)$ as in \eqref{atan}
$$
\Delta u(\cdot;\eps_n) \rightharpoonup \int_{\R} \xi \, d \pi_{(t,\mx)}(\xi) \ \ {\rm in} \ \ {\cal D}'([0.T]\times \Omega) \ \ {\rm as} \ \ n\to \infty.
$$ Comparing the last two relations, we reach to \eqref{connection}. The final statement in the theorem simply follows by substituting $\pi_{(t,\mx)}(\xi)=\delta(\xi-\Delta u(t,\mx))$ in \eqref{atan}.
\end{proof}

\subsection{Equation \eqref{che} with Dirichlet and Neumann conditions}

We shall now briefly discuss how to formulate a solution concept which would incorporate Dirichlet conditions \eqref{dirichlet} or Neumann conditions \eqref{neumann}, and how to prove well-posedness of the equation in the framework of such a concept.  To this end, we need to adapt the approximation procedure as follows.

We need the characteristic function $\chi_\eps$ of the $\eps^{1/2}$-interior of $\Omega_{\eps^{1/2}}=\{\mx \in \Omega:\, {\rm dist}(\mx,\pa \Omega)>\eps^{1/2} \}$:
$$
\chi_\eps(\mx)=\begin{cases}
1, & \mx\in \Omega_{\eps^{1/2}}\\
0, & \mx\notin \Omega_{\eps^{1/2}}
\end{cases}.
$$

We note that the support of the convolution $\chi_\eps\star \rho_\eps$, with $\rho_\eps$ given in \eqref{molif1}, is still compact in $\Omega$ for $\eps$ small enough:
\begin{equation}
\label{support}
{\rm supp}\chi_\eps\star \rho_\eps \subset \Omega, \ \ \eps<<1.
\end{equation}  Now, we consider the regularized equation 
\begin{align}
\label{regularized-1}
\pa_t u(\,\cdot\,;\eps) &= \Delta \Big(-\nu \big(\chi_\eps\, \arctan(\Delta u(\,\cdot \,;\eps) \star_\mx \rho_\eps)|\nabla u(\,\cdot \,;\eps)|\big)\star_\mx \rho_\eps - \mu \Delta u(\,\cdot\,;\eps) \Big) + \lambda(u_0 - u(\,\cdot\,;\eps))\\& :=\Delta \Big(-\nu \big(\chi_\eps \,\arctan(\Delta u_\eps)|\nabla u|\big)_\eps - \mu \Delta u \Big) + \lambda(u_0 - u),
\nonumber
\end{align} and augment it by the initial condition \eqref{ic}, and Dirichlet \eqref{dirichlet} or Neumann \eqref{neumann} boundary data. We note that presence of the cutoff function $\chi_\eps$ enables us to apply integration by parts as in the case of equation \eqref{regularized} despite inappropriate behavior of $u(\,\cdot\,;\eps)$ on the boundary.

Let us focus on the case of Dirichlet boundary data (the Neumann initial data are considered analogically). To this end, we look for an approximate solution to \eqref{regularized} in the form
\begin{equation}
\label{span}
u_n(t,\mx;\eps)=\sum\limits_{k=1}^n \alpha^n_k(t;\eps)e_k(\mx),
\end{equation} where $\alpha^n_k(t;\eps) \in C^1([0,T])$ are unknown functions and $e_k(\mx)$, $k\in {\bf N}$, are eigenvectors of the Laplace operator with zero boundary data i.e. they solve
\begin{equation}
\label{lapl}
\Delta e_k=\lambda_k e_k, \qquad e_k\big|_{\pa \Omega}=0.
\end{equation} Note that $e_k$, $k\in {\bf N}$, also satisfy $\Delta e_k\big|_{\pa \Omega} =\lambda_k e_k\big|_{\pa \Omega}=0$ which makes the approximation \eqref{span} consistent with boundary data \eqref{dirichlet}. Then, using the apriori estimates as in the proof of Theorem \ref{T1} together with the Aubin-Lions lemma, we are able to obtain the strong $L^2([0,T];H^1(\Omega))$-convergence of the sequence $(u_n)$ and the limit will be the weak solution to \eqref{regularized-1}, \eqref{ic}, \eqref{dirichlet}. 

To proceed, we introduce the space
$$
W_n=\{\varphi \in  C^1([0,T];H^2(\Omega)): \, \exists n\in  {\bf N} \ \ \varphi(t,\mx)\in Span\{e_k \}_{k=1,\dots,n} \},
$$ where $Span\{e_k \}_{k=1,\dots,n}$ is the standard span of the basis $\{e_k \}_{k=1,\dots,n}$ of the form \eqref{span} consisting of solutions to \eqref{lapl}.

Then, we let $\epsilon \to 0$ along an appropriate subsequence in \eqref{regularized-1} in the sense of distributions, and we will end up with \eqref{regularized-limit}. A similar situation arises with the entropy conditions, and one can derive them in the same way as when considering \eqref{regularized-1}. Consequently, an admissible solution to \eqref{che} with \eqref{dirichlet} will be unique.

A natural question here is how to distinguish between solutions to \eqref{che} with different initial conditions. This is done by making an appropriate choice of the test function. We introduce the following definition of an entropy solution to \eqref{che}, \eqref{ic}, and \eqref{dirichlet}. 

\begin{definition}
\label{def-admissibility-D} We say that the function $u\in L^2([0,T];H^2(\Omega))$ represents an entropy solution to problem \eqref{che}, \eqref{ic}, \eqref{dirichlet} if
the following inequality holds for every $\varphi \in W_n$ satisfying \eqref{dirichlet} and almost every $t\in [0,T]$:
\begin{equation}
\label{solu-3}
\begin{split}
  & \int_{\Omega}\frac{1}{2}( u(t,\mx)-\varphi(t,\mx))^2   d\mx -\int_{\Omega}\frac{1}{2}( u_0(\mx)-\varphi(0,\mx))^2   d\mx  \\&+ \int_0^t \!\!\int_{\Omega}(u-\varphi)\pa_t \varphi d\mx dt' \leq \int_0^t\!\!\int_\Omega  \Big(\nu \arctan(\Delta \varphi)|\nabla u|\big) \Delta(u-\varphi) +\mu \Delta \varphi \Delta(u-\varphi) \Big)d\mx dt' \\&+ \int_0^t\int_\Omega \lambda(u_0 - u)(u-\varphi)d\mx dt'.
\end{split}    
\end{equation}

\end{definition} The completely same procedure works in the case of conditions \eqref{neumann}, but we need to replace the eigenvalues $e_k$ by the ones solving the Neumann problem

\begin{equation*}
\Delta e_k=\lambda e_k, \qquad \frac{\pa e_k}{\pa \vec{n}}\big|_{\pa \Omega}=0.
\end{equation*} Regarding existence and uniqueness proofs, they go along the lines of Theorem \ref{T1} and Theorem \ref{T-uniqueness}. The proof of Theorem \ref{T1} for the problem \eqref{che}, \eqref{ic}, \eqref{dirichlet} is completely the same as the one for \eqref{che}, \eqref{ic}, \eqref{bc}.

We also note that in the uniqueness proof, it is enough to replace the convolutions $u^\eps$ and $v^\eps$ appearing there by $N$-dimensional approximations of the entropy solutions $u$ and $v$ of the form \eqref{span} and then to let the approximation parameter $N$ tend to infinity. 


\section{Numerical method}

In this section, we implement a fast solver for the Shock Filter CHE based on the convexity splitting idea introduced by Eyre \cite{Eyre98}. This method was later refined by Vollmayr-Lee and Rutenberg \cite{Vol03} and applied to image inpainting problems by Bertozzi, Esedoglu and Gillette \cite{Bert07, Bert07d}.

Let us recall that the CHE is a gradient flow in $H^{-1}$ with the energy
\begin{align}
	\mathcal{E}_1[u] &= \int_{\Omega}\left(\dfrac{\varepsilon}{2}|\nabla u|^2 + \dfrac{1}{\eps}H(u)\right)  d\mx,
\end{align}
and the fidelity term can be obtained from the $L^2$ gradient flow
\begin{align}
	\mathcal{E}_2[u] &= \int_{\Omega \setminus\omega} \lambda_0(u_0-u)^2 d\mx.
\end{align}
Here, $h(u)$ is a real-valued differentiable function such that 
\begin{align}\label{hH}
h(u) =  \nabla_{H^{-1}}\left(\int_\Omega H(u)d\mx\right) =  -|\nabla u| \arctan({\Delta u}),
\end{align}
where $\nabla_{H^{-1}}(\cdot)$ means the gradient descent with respect to the $H^{-1}$ norm. 
It is worth noting that only the function $h$ is needed for the numerical scheme, and the formula \eqref{hH} is provided for clarity, as $H$ is not used in the proposed numerical method. For the simulations, we follow the practice where $\nu \cdot \mu = 1$ ($\nu= \epsilon$ and $\mu = 1/\eps$), which experimentally showed the best inpainting results for CHE with a double-well potential.

As already mentioned, the considered equation is neither a gradient flow in $L^2$ nor in $H^{-1}$, but we can still adopt the idea of convexity splitting to obtain fast inpainting results. Namely, the idea is to apply the convexity splitting on $\mathcal{E}_1$ and $\mathcal{E}_2$ to split each of the energies into a convex and a concave part and then construct a semi-implicit numerical scheme. The concave part is treated explicitly, and the convex terms are treated implicitly in time. More precisely, if we write the energies as: 
\begin{align}
	\mathcal{E}_1[u] &= \mathcal{E}_{11}[u] - \mathcal{E}_{12}[u],\\  
    \mathcal{E}_2[u] &= \mathcal{E}_{21}[u] - \mathcal{E}_{22}[u].
\end{align}
The resulting convexity splitting scheme is unconditionally energy stable, unconditionally solvable, and converges optimally in the energy norm if all the energies $\mathcal{E}_{ij}$, $i,j = {1,2}$ are convex. 
To make the energies convex, we employ the following decomposition:
\begin{align}\label{e11}
	\mathcal{E}_{11}[u] &= \int_{\Omega} \left(\frac{\varepsilon}{2}\qty|\nabla u|^2 + \frac{C_1}{2}|u|^2\right)d\mx,\\  \label{e12}
    \mathcal{E}_{12}[u]&= \int_{\Omega}  \left(-\frac{1}{\varepsilon}{H(u)}+ \frac{C_1}{2}|u|^2\right)d\mx,\\\label{e21}
    \mathcal{E}_{21}[u] &=\int_{\Omega \setminus\omega}\frac{C_2}{2}|u|^2d\mx,\\ \label{e22}
    \mathcal{E}_{22}[u] &=\int_{\Omega \setminus\omega}  \left(-\lambda_0(u_0-u)^2+\frac{C_2}{2}|u|^2 \right)d\mx.
\end{align}
The constants $C_1, C_2 > 0$ need to be chosen large enough to ensure that the energies $\mathcal{E}_{ij}[u]$ are convex. Obviously, $\mathcal{E}_{11}[u]$, $\mathcal{E}_{21}[u]$ are convex independently on $C_1$ and $C_2$, and if $C_2 > \lambda_0$ then $\mathcal{E}_{22}[u]$ is convex. Furthermore, if we take into account \eqref{hH}, for the simulations, we can always find $C_1(\varepsilon)$ large enough so that $\mathcal{E}_{12}[u]$ is convex.

The splitting defined with \eqref{e11},\eqref{e12},\eqref{e21},\eqref{e22}, yields the following time-stepping scheme
\begin{align}
    \label{TS-energies}
	\frac{u^{n+1}-u^n}{\Delta t} = -\nabla_{H^{-1}}\qty\left(\mathcal{E}_{11}\qty[u^{n+1}]-\mathcal{E}_{12}\qty[u^{n}]\right)-\nabla_{L^2} \qty(\mathcal{E}_{21}\qty[u^{n+1}] - \mathcal{E}_{22}\qty[u^n]),
\end{align}
where $n$ denotes the time step and $\Delta t$ denotes step size. 
Inserting $\mathcal{E}_{ij}$ into \eqref{TS-energies} translates to the following numerical scheme:
\begin{align}\label{DS-equations}
	\qty(1+C_2\Delta t)u^{n+1} + & \,\varepsilon\,\Delta t\,\Delta ^2 u^{n+1} - C_1\,\Delta t\,\Delta u^{n+1} \nonumber\\
	& = \frac{\Delta t}{\varepsilon} \Delta h(u^n) - C_1\,\Delta t\,\Delta u^n + \lambda\,\Delta t\,(u_0-u^n) + \qty(1+C_2\Delta t)u^n.
\end{align}
Because of the significant simplification of Laplacian and Bi-Laplacian operators in the Fourier space, described hereafter, we will implement the Fourier spectral method for spatial discretization.

After computing the Fourier transform, rearranging \eqref{DS-equations} results in the following iteration scheme: 
\begin{equation}\label{scheme}
\begin{split}
    (\widehat{u^{n+1}})_{k,l} = \frac{\widehat{u^n}_{k,l}+\Delta t\qty\left(\dfrac{1}{\varepsilon} K_{k,l}\cdot (\widehat{h(u^n)})_{k,l} - C_1 K_{k,l}\cdot \widehat{u^n}_{k,l} + \widehat{\qty(\lambda(u_0-u^n))}_{k,l}+C_2 \widehat{u^n}_{k,l}\right)}{1+ C_2+\varepsilon K_{k,l}^2-C_1 K_{k,l}}.
\end{split}
\end{equation}
To obtain $u^{n+1}$ in the direct space we simply perform the inverse discrete Fourier transformation on $\widehat{u^{n+1}}$.
Finally, let us just note that in the case of CHE, it is easy to obtain first-order schemes that are unconditionally energy stable using the convexity splitting \eqref{e11},\eqref{e12},  and similar extensions could be done for \eqref{e21}, \eqref{e22}. However, given that modified CHE cannot be derived as a gradient flow in Hilbert space, one cannot extend the approach from \cite{Pe19, Shin17, Sh21} to our case. 
Nevertheless, as we will demonstrate in the next section, the inpainting results obtained with the proposed approach are obtained fast and of high quality as long as constants $C_1$ and $C_2$ are large enough to ensure the convexity of the energies $\mathcal{E}_{ij}.$

\section{Results}

In this section, we explore the potential applications of the proposed Shock Filter Cahn-Hilliard equation for image inpainting problems. We focus on the inpainting of binary images containing conventional shapes, such as stripes and crosses, and compare these results with the ones obtained with the CHE equipped with the conventional double-well potential. Throughout the provided examples, the designated inpainting area is marked in grey.

In all our simulations, we utilize the two-scale approach delineated in \cite{Bert07d}. This method has shown efficacy in bridging edges over extensive inpainting regions when applied to the CHE. In the first step, the inpainting is carried out with a larger value of the regularizing parameter, resulting in the topological reconnection of shapes and edges that may have been smeared by diffusion. The second step utilizes the results from the first step and continues with a much smaller value of the regularizing parameter to sharpen the edges after reconnection. 

All programs were written in Matlab, in accordance with the code and notation written by Parisotto and Schönlieb \cite{Simo20}, which served as the reference code for the simulations involving the CHE with double well potential.
All simulations were executed on a standard desktop computer. For simplicity, we scaled the images to grayscale so that their intensities lie within the range $[0, 1]$. Unless otherwise specified, all time values are presented as the number of iterative steps taken by the numerical algorithm, with $t = n$ indicating the $n-$th step of the numerical scheme. Moreover, all methods in our evaluation use optimized parameters obtained through the exhaustive enumeration method.

\subsection{Inpainting of the standard binary shapes}
\begin{figure}[H]
\centering  
\subfigure[Original image]{\label{1a}\includegraphics[width=4cm]{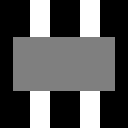}}
\subfigure[CH equation]{\label{1b}\includegraphics[width=4cm]{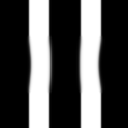}}
\subfigure[Proposed equation]{\label{1c}\includegraphics[width=4cm]{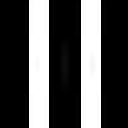}}
    \caption{Results of image inpainting for the image of stripes.}
\label{fig1}
\end{figure}

In Figure \ref{fig1}, we present a comparison between the proposed equation and the CHE equipped with the double-well potential for inpainting a standard binary image of stripes. The image domain is represented by a grey rectangle, which is larger than what is commonly used in the literature for this test image \cite{Bert07d, Garc18}. The comparison was based on the quality of the resulting images, as well as the presence of diffusive effects introduced by the CHE. 
Figure \ref{1b} was obtained after 8000 iterations, where the first 4000 iterations were done with $\eps = 100$, and then $\eps = 1$. Furthermore, besides the irregularly reconstructed inpainting domain, we can observe the diffusive effects in the regions containing black-to-white or white-to-black transitions (edges). On the other hand, in Figure \ref{1c} we have recovered the sharp edges and the natural continuation of the features of the image into the inpainting domain. The image was obtained after 10000 iterations, using the same two-scale approach with $\eps = 100$ and then $\eps = 2$. As discussed later, the inpainting process seems to slow down as the inpainting domain becomes barely detectable. Even though mean-squared error (MSE) values are known to be in discordance with the human eye, for this example, the image obtained by the CHE resulted in an MSE of 0.0172, while the image obtained by the proposed equation has an MSE of 0.0157.
 
\begin{figure}[H]
\centering  
\subfigure[Original image]{\label{2a}\includegraphics[width=4cm]{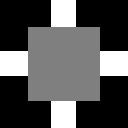}}
\subfigure[CH equation]{\label{2b}\includegraphics[width=4cm]{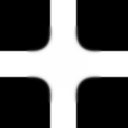}}
\subfigure[Proposed equation]{\label{2c}\includegraphics[width=4cm]{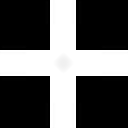}}
    \caption{Results of image inpainting for the image of a cross.}
\label{fig2}
\end{figure}

In Figure \ref{fig2}, we consider a cross-shaped image with a damaged region at the centre. The results of the inpainting process obtained with the CHE are presented in Figure \ref{2b} and are obtained after 8000 iterations, following the two-scale approach with $\eps = 100$, and then $\eps = 1$.
Let us note that these results are consistent with the results obtained in previous studies (e.g., Figure 3.1 in \cite{Bert11}, Figures 1 and 2 in \cite{Bert07}, or Figure 3 in \cite{Bosc14}). 
In the inpainted image produced by the proposed equation in Figure \ref{2c}, we observe that it effectively restores the missing information while preserving the edges and extending the image features in a natural manner. In comparison with Figure 4 in \cite{Nova22}, the newly formed edges in the centre of the image seem more natural and can be obtained without coupling the equation with the thresholding operation. This image was obtained after 30000 iterations with $\eps = 100$ and then $\eps = 2$.

\begin{figure}[H]
\centering  
\subfigure[After 100 iterations]{\label{3a}\includegraphics[width=3.8cm]{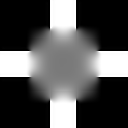}}
\subfigure[After 1000 iterations]{\label{3b}\includegraphics[width=3.8cm]{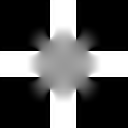}}
\subfigure[After 4000 iterations]{\label{3c}\includegraphics[width=3.8cm]{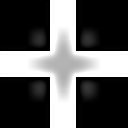}}
\subfigure[After 8000 iterations]{\label{3d}\includegraphics[width=3.8cm]{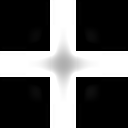}}
    \caption{Snapshots of the evolution of the inpainting process for the image of a cross (as in Figure \ref{fig2}).}
\label{fig3}
\end{figure}

Figure \ref{fig3} presents a sequence of snapshots demonstrating the progression of the inpainting process when the maximal number of iterations is set to 100, 1000, 4000, and 8000 iterations. It can be observed that the initial phase of the inpainting process is characterized by a rapid recovery of the image features. However, as the inpainting domain diminishes, the process tends to become slower, and the area of image restoration between iteration numbers 4000 and 8000 (Figure \ref{3c} and Figure \ref{3d}) is relatively modest compared to the area recovered during the initial 1000 iterations (Figure \ref{3a} and Figure \ref{3b}). The final image was obtained after 30000 iterations and is displayed in Figure \ref{2c}.

\begin{figure}[H]
\centering  
\subfigure[Original image]{\label{4a}\includegraphics[width=3.8cm]{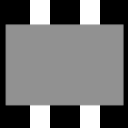}}
\subfigure[Resulting image (a)]{\label{4b}\includegraphics[width=3.8cm]{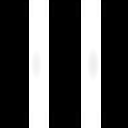}}
\subfigure[Original image]{\label{4c}\includegraphics[width=3.8cm]{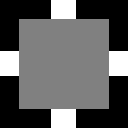}}
\subfigure[Resulting image (b)]{\label{4d}\includegraphics[width=3.8cm]{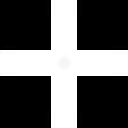}}
    \caption{Inpainting over large domains using the Shock Filter Cahn-Hilliard Equation.}
\label{fig4}
\end{figure}
The results of inpainting for a large inpainting domain are displayed in Figure \ref{fig4}, where one can observe that the image features have been retrieved without any apparent diffusive effects. However, it is noteworthy that the process required a substantial number of iterations to restore the entire image because the size of the inpainting domain constituted approximately $60\%$ of the image size. The same parameter $\eps$ was used as in Figures \ref{fig1} and \ref{fig2}. Figure \ref{4b} was obtained after 20000 iterations, while Figure \ref{4d} was obtained after 60000 iterations.

\subsection{Inpainting results for different choice of parameters}

\begin{figure}[H]
\centering  
\subfigure[Original image]{\label{5a}\includegraphics[width=3.0cm]{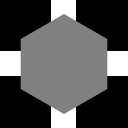}}
\subfigure[After 200 iterations]{\label{5b}\includegraphics[width=3.0cm]{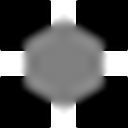}}
\subfigure[After 1000 iterations]{\label{5c}\includegraphics[width=3.0cm]{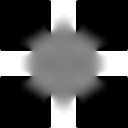}}
\subfigure[After 4000 iterations]{\label{5d}\includegraphics[width=3.0cm]{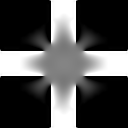}}
\subfigure[After 8000 iterations]{\label{5e}\includegraphics[width=3.0cm]{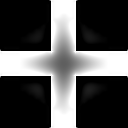}}

\subfigure[Original image]{\label{5f}\includegraphics[width=3.0cm]{img/cross_arctg_6_init.png}}
\subfigure[After 200 iterations]{\label{5g}\includegraphics[width=3.0cm]{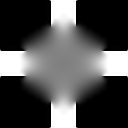}}
\subfigure[After 1000 iterations]{\label{5h}\includegraphics[width=3.0cm]{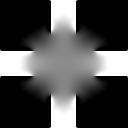}}
\subfigure[After 4000 iterations]{\label{5i}\includegraphics[width=3.0cm]{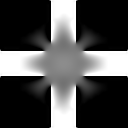}}
\subfigure[After 8000 iterations]{\label{5j}\includegraphics[width=3.0cm]{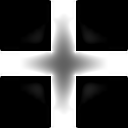}}

\subfigure[Original image]{\label{5k}\includegraphics[width=3.0cm]{img/cross_arctg_6_init.png}}
\subfigure[After 200 iterations]{\label{5l}\includegraphics[width=3.0cm]{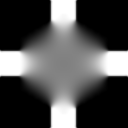}}
\subfigure[After 1000 iterations]{\label{5m}\includegraphics[width=3.0cm]{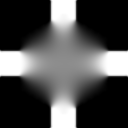}}
\subfigure[After 4000 iterations]{\label{5n}\includegraphics[width=3.0cm]{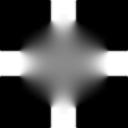}}
\subfigure[After 8000 iterations]{\label{5o}\includegraphics[width=3.0cm]{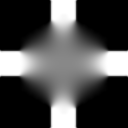}}
    \caption{Inpainting over large and complex domains using the Shock Filter Cahn-Hilliard Equation for various parameter choices.}
\label{fig5}
\end{figure}

Figure \ref{fig5} depicts a binary image of a cross with a hexagonal inpainting domain located in the centre of the image. Each row exhibits the temporal evolution of the inpainting process for a different choice of parameters $\eps$ with a maximal number of iterations set to 200, 1000, 4000, and 8000 iterations, respectively. We can observe that the changes made inside the inpainting domain are similar for the first two experiments, that is, if the process converges, then it converges following the comparable paths.
In the first experiment (Figures \ref{5a}-\ref{5e}), the parameters used were $\eps = 5$ initially, followed by $\eps = 2$. The second row (Figures \ref{5f}-\ref{5j}), involved the use of $\eps = 100$ initially, followed by $\eps = 2$, which was the same as in the first experiment. In the third experiment (Figures \ref{5k}-\ref{5o}), we used $\eps = 100$ initially, and then $\eps =3$, and we can notice that for this choice of parameters the inpainting domain seems to remain unchanged after the first 200 iterations.


\subsection{Limitations}
Upon testing the Shock Filter CHE, we observed that it exhibited suitable performance in the context of inpainting when the underlying image comprised straight lines, edges, and corners. However, for images featuring numerous flow-like structures, such as fingerprint images, the performance was modest, and it is advisable to rely on the classical Cahn-Hilliard equation in such cases. We believe that this could be resolved by replacing the Laplacian as the edge classifier for a structure tensor as in the work of Weickert \cite{Weic03} to guarantee that this equation creates shocks orthogonal to the flow direction of the pattern.
It is noteworthy that the selection of appropriate parameters poses a significant challenge for all PDE models, and this issue is equally relevant for this model, especially in the context of nonstandard images and large domains. For example, in Figure \ref{fig5} one can notice that a small change in the $\eps$ in the last experiment resulted in the incomplete inpainting of the unknown region.
We have also observed that the CHE required fewer iterations to produce the inpainted images compared to the proposed method. For instance, Figure \ref{2b} was obtained after 8000 iterations of the two-step process, whereas Figure  \ref{2c} was obtained after 30000 iterations. 
Overall, these findings suggest that the proposed equation is a viable alternative to the CHE for image inpainting, as it produces images with fewer diffusive effects, albeit requiring more iterations than the CHE.
In short, these results highlight the trade-off between computational efficiency and the quality of the reconstructed images when selecting an appropriate nonlinear classifier $H$.

\section*{Conclusions and final remarks}
Inspired by the inpainting methods based on the Cahn-Hilliard equation, we have introduced the fully nonlinear fourth-order PDE in which the image reconstruction is driven by the diffusion joined by the nonlinear term belonging to the class of morphological image enhancement methods. 
Our primary objective was to establish the well-posedness of this equation and investigate its potential for inpainting tasks. 

To achieve the latter, we devised a robust adaptation of the Kruzhkov entropy admissibility approach which enabled us to address the challenging non-linearity present in the equation and to demonstrate the existence and stability of an appropriate class of solutions.

Our work not only contributes to the field of inpainting but also sheds light on the broader topic of non-linear partial differential equations in image processing. By establishing the well-posedness of our equation, we have taken a significant step towards enhancing the understanding of complex non-linear systems in this context.

The obtained results provide a solid theoretical foundation for the practical application of this inpainting model. To this end, we have utilized the idea of convexity splitting to derive a numerical scheme for practical implementation. Using the testing binary images that contain straight lines or edges in the inpainting domain, we have demonstrated that the proposed morphological classifier naturally extends image structures and preserves image sharpness, leading to superior results compared to a standard Cahn-Hilliard equation with a double-well potential. 
For binary shapes that contain a lot of curves (such as fingerprint image or zebra), we have noticed that the introduced equation shows more modest results and it seems to be more sensitive to changes in parameters as compared to the classical CHE.

In conclusion, we have presented a well-posed formulation of the Shock Filter Cahn-Hilliard equation for image inpainting and conducted numerical simulations in order to solidify both its theoretical foundation and demonstrate its potential for practical applications.
We hope that our result will contribute toward enriching the theoretical landscape of inpainting models and providing a deeper perspective on the role of fully non-linear PDEs in image processing.

 \section*{Acknowledgement}
This work was supported in part by the Croatian Science Foundation under project number HRZZ-MOBODL-2023-08-7617 and IP-2022-10-7261 Analysis of Partial Differential Equations and Shape Optimization (ADESO), the Austrian Science Foundation (FWF) Stand Alone Project number P-35508-N, and the Croatian-Austrian bilateral project Mathematical Aspects of Granular Hydro-dynamics – Modelling, Analysis, and Numerics.

\section*{Data Availability}

Data sharing is not applicable to this article as no datasets were generated or analyzed during the current study.

\end{document}